\documentclass[11pt,a4paper]{article}
\usepackage[utf8]{inputenc}
\usepackage[T1]{fontenc}
\usepackage{lmodern,textcomp,amsmath,amssymb,mathtools,amsthm,braket,authblk,graphicx,caption,subcaption,csquotes}
\usepackage[english]{babel}
\usepackage[a4paper]{geometry}
\newtheorem{theorem}{Theorem}[section]
\newtheorem{corollary}[theorem]{Corollary}
\newtheorem{lemma}[theorem]{Lemma}
\newtheorem{proposition}[theorem]{Proposition}
\theoremstyle{definition}
\newtheorem{definition}[theorem]{Definition}

\newtheorem{hypothesis}[theorem]{Hypothesis}
\theoremstyle{remark}
\newtheorem{remark}[theorem]{Remark}

\newcommand{\hilb}{\mathcal{H}}
\newcommand{\e}{\mathrm{e}}

\renewcommand{\Re}{\operatorname{Re}}
\newcommand{\ii}{\mathrm{i}}

\makeatletter
\def\smalloverbrace#1{\mathop{\vbox{\m@th\ialign{##\crcr\noalign{\kern3\p@}%
	\tiny\downbracefill\crcr\noalign{\kern3\p@\nointerlineskip}%
	$\hfil\displaystyle{#1}\hfil$\crcr}}}\limits}
\makeatother

\usepackage[unicode=true,
bookmarks=true,bookmarksopen=false,
breaklinks=false,pdfborder={0 0 0},colorlinks=true]
{hyperref}
\usepackage{xcolor}
\definecolor{cblue}{rgb}{0.16, 0.32, 0.75}
\definecolor{cred}{rgb}{0.7, 0.11, 0.11}
\hypersetup{%
	,linkcolor=cred
	,citecolor=cblue
	,urlcolor=black
}

\usepackage[
backend=biber,
style=numeric-comp,
giveninits=true,
maxbibnames=299,
natbib=true,
doi=false,
isbn=false,
url=false,
sorting=nyt
]
{biblatex}
	\bibliography{controlwalls}
\renewbibmacro{in:}{}

\newcommand\blfootnote[1]{%
	\begingroup
	\renewcommand\thefootnote{}\footnote{#1}%
	\addtocounter{footnote}{-1}%
	\endgroup
}

\let\originalleft\left
\let\originalright\right
\renewcommand{\left}{\mathopen{}\mathclose\bgroup\originalleft}
\renewcommand{\right}{\aftergroup\egroup\originalright}

\newcommand{\refsigma}{$($\hyperref[eq:schro]{$\Sigma$}$)$}
\newcommand{\reftildesigma}{$($\hyperref[eq:schro2]{$\tilde\Sigma$}$)$}

\title{\textbf{On global approximate controllability of a quantum particle in a box by moving walls}}

\author[$\hspace{0cm}$]{Aitor Balmaseda$^{1,2,}$\footnote{abalmase@math.uc3m.es}}
\author[$\hspace{0cm}$]{Davide Lonigro$^{3,4,}$\footnote{davide.lonigro@ba.infn.it}}
\author[$\hspace{0cm}$]{Juan Manuel Pérez-Pardo$^{1,5,}$\footnote{jmppardo@math.uc3m.es}}

\affil[$1$]{\small Departamento de Matemáticas, Universidad Carlos III de Madrid, Avda. de la Universidad 30, 28911  Madrid, Spain}
\affil[$2$]{\small Department of Mathematics, Faculty of Nuclear Sciences and Physical Engineering, Czech Technical University in Prague, Trojanova 13, 12000 Prague 2, Czechia}
\affil[$3$]{\small Dipartimento di Fisica and MECENAS, Università di Bari, I-70126 Bari, Italy}
\affil[$4$]{\small Istituto Nazionale di Fisica Nucleare, Sezione di Bari, I-70126 Bari, Italy}
\affil[$5$]{\small Instituto de Ciencias Matemáticas (CSIC -- UAM -- UC3M -- UCM) ICMAT, C/ Nicolás Cabrera 13--15, 28049 Madrid, Spain.}

\begin{document}
	
\maketitle
\vspace{-.5cm}

\begin{abstract}
	We study a system composed of a free quantum particle trapped in a box whose walls can change their position. We prove the global approximate controllability of the system. That is, any initial state can be driven arbitrarily close to any target state in the Hilbert space of the free particle with a predetermined final position of the box. To this purpose we consider weak solutions of the Schrödinger equation and use a stability theorem for the time-dependent Schrödinger equation.
\end{abstract}

\vspace{-0.5cm}

\section{Introduction}

The recent developments in quantum technology foster the demand for a thorough theoretical study of the controllability properties of quantum systems. This attracts the interest of both mathematicians and physicists. The theory of quantum control of finite dimensional quantum systems has achieved a great maturity and has been used successfully in the development of quantum technologies~\cite{d2021introduction}. These developments are partly due to the well established theory of geometric control~\cite{agrachev2004control,jurdjevic1997geometric}. However, many quantum systems relevant for the applications are infinite-dimensional in nature and the use of finite-dimensional control techniques on them leads unavoidably to the introduction of truncation errors. Moreover, infinite-dimensional quantum systems allow for new types of control systems that would not be possible in finite dimension~\cite{ibort2009quantum,balmaseda2019quantum,perez2015boundary}

The results of finite-dimensional quantum control cannot be carried over straightforwardly to infinite-dimensional systems~\cite{mirrahimi2004} and new approaches are necessary. Even the notions of controllability have to be revised in the infinite-dimensional setting. For instance, the negative results~\cite{ball1982,turinici2000} show that exact controllability in the complete space of quantum states is not possible. Two possible ways to avoid these obstructions have been found. One of them is to look for solutions of the control problem in regular dense subspaces of the Hilbert space. For instance, one can achieve local exact controllability of the one-dimensional Schrödinger equation by an electric field if one considers only states in higher order Sobolev spaces~\cite{beauchard2005,morancey2014,morancey2015}. The non-linear Schrödinger equation is also locally exactly controllable on regular dense subspaces~\cite{beauchard2010}. This approach has also been used to obtain controllability results in situations where the base manifold for the Schrödinger operator is of dimension higher than one, but these are more scarce and limited~\cite{beauchard2009,moyano2017}. 

Another option is to give up exact controllability and look for approximate controllability, that is, the possibility of driving any initial state in a neighborhood of any target state with the desired precision. In this way one can obtain controllability results in larger domains. A successful and general approach is the one developed during the last decade ~\cite{chambrion2009controllability,boscain2012weak,boussaid2013weakly} in which approximate controllability is proven under mild assumptions for bilinear quantum control systems. There, techniques from geometric control theory have been extended to the infinite-dimensional case. These results can be applied to quantum control systems whose Schrödinger operator is defined over base manifolds of any dimension, for instance to the control of molecules. This latter approach has the drawback that only piecewise constant controls are admitted and there are interesting applications for which they are not suitable~\cite{ervedoza2009}. In particular, situations in which the Hamiltonian is unbounded and with time-dependent domain, as will be considered here, are not well-posed if the controls are piecewise constant. 

In this article we address the problem of controlling the state of a quantum particle confined in a one-dimensional box by moving the walls of the box. We shall take the second approach described above and prove that this quantum control system is approximately controllable, that is, any initial state can be driven arbitrarily close to any target state, cf.\ Theorem~\ref{thm:main}. This problem has been considered previously in the literature~\cite{munier1981,band2002, carrasco2017controlling,duffin2019controlling} and local exact controllability results have been obtained already. In~\cite{rouchon2003control,beauchard2006controllability} it is proven that one can achieve local exact controllability between sufficiently small regular neighborhoods of the eigenstates by rigidly moving the box, i.e.\ changing the position of the box but maintaining a fixed length. It was proven also under similar assumptions~\cite{coron2006} that control cannot be achieved in small times. Local exact controllability around eigenstates by only moving one wall was proven in~\cite{beauchard2008} and positive results for the non-linear Schrödinger equation have been obtained more recently in~\cite{beauchard2015}.

In the previous cases a suitable change of variables is used that maps the problem with varying domain into another equivalent problem with fixed domain. In this article we shall also use a convenient but different change of coordinates, introduced in~\cite{dimartino2013quantum}, that allows us to consider more general movement of the box. The situations with rigid movement of the box and with one fixed wall and the other moving are recovered as particular cases. In addition, we can also treat the case of a purely stretching or shrinking box. Approximate controllability is proven in this case as well, but only with sectors of fixed parity.

The novelty with respect to previous approaches is that we are able to prove global approximate controllability and target any state of the system, not just neighborhoods of the eigenstates. We consider more general movement of the walls, but global approximate controllability is also obtained for the particular case of one wall fixed, thus improving previous results in the literature. As a side result we have obtained approximate controllability results for interactions that are different from electric fields, namely the dilation operator. This is achieved by considering weak solutions of the Schrödinger equation, which allows us to use a stability theorem, Theorem~\ref{thm:abstract2}, that we can use to extend the results in~\cite{chambrion2009controllability,boscain2012weak} to admit controls that are not piecewise constant. 

This work is organized as follows. The main results of the work are presented and discussed in Section~\ref{sec:main}. In Section~\ref{sec:prelimin} we gather some known results about the existence of unitary propagators generated by time-dependent Hamiltonians and their stability properties, along with known abstract results on the approximate controllability of such systems, and use them to show existence of solutions of the Schrödinger equation associated with a quantum particle in a box with moving walls. Finally, Section~\ref{sec:result} is devoted to the proof of approximate controllability. Some concluding remarks and outlooks are collected in Section~\ref{sec:conclusions}.

\section{Preliminaries and main results}\label{sec:main}

Throughout this work, the symbol $\hilb$ will denote a complex separable Hilbert space, with associated scalar product $\Braket{\cdot,\cdot}$ antilinear at the left, and norm $\|\Psi\|=\sqrt{\Braket{\Psi,\Psi}}$; the domain of an unbounded linear operator $A$ on $\hilb$ will be denoted by $\mathcal{D}(A)$. The adjoint of $A$ is denoted by $A^\dag$.	Finally, given a real interval $I\subset\mathbb{R}$ and an integer $p\in\mathbb{N}$, the space of functions $f:I\rightarrow\mathbb{C}$ that admit $d$ continuous derivatives will be denoted by $\mathrm{C}^d(I)$, while the space of functions $f:I\rightarrow\mathbb{C}$ that admit $d$ \textit{piecewise} continuous derivatives will be denoted by $\mathrm{C}^d_{\rm p}(I)$; that is, $f\in \mathrm{C}^d_{\mathrm{p}}(I)$ if there exists a finite partition of the interval $I=\sqcup_{i=1}^n I_i$ such that the restriction of $f$ to the interior of each subinterval is $d$ times boundedly continuously differentiable; $f|_{\dot{I}_i}=\mathrm{C}^d(\dot{I}_i)$ and $\smash{\sup_{x\in\dot{I}_i} \{f^{(k)}(x)\}}<\infty$, $k=0,\dots,d$.

\subsection{Dynamics of a particle in a moving box}

Let us start by revising some abstract notions.
\begin{definition}\label{def:unitary}
	Let $I\subseteq\mathbb{R}$ be a compact real interval. A \textit{unitary propagator} is a collection $\{U(t,s)\}_{t,s\in I}$ of operators on $\hilb$ such that
	\begin{itemize}
		\item[(i)] for all $t,s\in I$, $U(t,s)$ is unitary;
		\item[(ii)] for all $t,s,r\in I$, $U(t,s)U(s,r)=U(t,r)$;
		\item[(iii)] the function $(t,s)\in I\times I\mapsto U(t,s)$ is jointly strongly continuous.
	\end{itemize}
\end{definition}
Notice that (i) and (ii) imply $U(t,s)^{-1}=U(s,t)$ and $U(t,t)=\mathbb{I}$ for all $t,s$.
\begin{definition}
	A \textit{time-dependent Hamiltonian} on $\hilb$ is a family of densely defined self-adjoint operators $\{H(t):\mathcal{D}(H(t))\rightarrow\hilb\,|\,t\in I \subset \mathbb{R}\}$. 
\end{definition}
We remark that, in general, $\mathcal{D}(H(t))\subset\hilb$ may depend non-trivially on time. Hereafter, with an abuse of notation, we shall indicate such objects by $U(t,s)$, $t,s\in I$, and $H(t)$, $t\in I$, whenever no ambiguities arise from this choice.
\begin{definition}
	Let $H(t)$, $t\in I$, be a time-dependent Hamiltonian. 
	Consider the equation
	\begin{equation}\label{eq:schro0}
		\ii\frac{\mathrm{d}}{\mathrm{d}t}\Psi(t)=H(t)\Psi(t).
	\end{equation}
	We say that a unitary propagator $U(t,s)$, $t,s\in I$, is a \textit{(strong) solution of the Schrödinger equation} if, for all $\Psi_0\in\mathcal{D}(H(s))$, the function $t\in I\mapsto\Psi(t):=U(t,s)\Psi_0$ solves Eq.~\eqref{eq:schro0} with initial condition $\Psi(s)=\Psi_0$.
\end{definition}
In the time-independent case, i.e.\ $H(t)\equiv H$, Stone's Theorem ensures that $U(t,s)=\e^{-\ii (t-s)H}$ is a strong solution of the Schrödinger equation. The situation in the time-dependent case, when the domains of the operators may change with time, is much more involved. In particular, one needs to take into account the highly non-trivial condition $U(t,s)\mathcal{D}(H(s))=\mathcal{D}(H(t))$, without which Eq.~\eqref{eq:schro0} is ill-defined. Sufficient conditions for the existence of strong solutions can be found in~\cite{Kisynski1964,Simon1971,balmaseda2021controllability,balmaseda2021schrodinger,balmaseda2021thesis}. A typical class of time-dependent Hamiltonians, which is particularly important for the scope of this work, is given by Hamiltonians of the form
\begin{equation}\label{eq:sum}
	H(t)=\sum_{i=0}^\nu f_i(t)H_i,
\end{equation}
with $\{H_0,H_1,\dots,H_\nu\}$ being a collection of symmetric operators, and $\{f_0,f_1,\dots,f_\nu\}$ a collection of real-valued functions on a compact time interval $I\subset\mathbb{R}$. Under the assumptions given in Theorem~\ref{thm:abstract2}, which in particular involve the time-independence of the \textit{form domain}, the time-dependent Hamiltonian of Eq.~\eqref{eq:sum} admits a unique solution satisfying useful stability properties, cf.~\cite{balmaseda2021controllability} and Theorem~\ref{thm:abstract2}.

In many cases of practical interest, it suffices to consider solutions of the Schrödinger equation in the weak sense:
\begin{definition}\label{def:weaksolution}
	Let $H(t)$, $t\in I$, be a uniformly semibounded time-dependent Hamiltonian with constant form domain $\hilb^+$,
	and let $\Phi\in\hilb^+$. Consider the equation	
	\begin{equation}\label{eq:schro_weak}
		\ii\frac{\mathrm{d}}{\mathrm{d}t}\Braket{\Phi,\Psi(t)}=h_t\left(\Phi,\Psi(t)\right),
	\end{equation}
	with $h_t(\cdot,\cdot)$ being the sesquilinear form uniquely associated with $H(t)$. We say that a unitary propagator $U(t,s)$, $t,s\in I$, is a \textit{weak solution of the Schrödinger equation} if, for all $\Psi_0\in\hilb^+$, the function $t\in I\mapsto \Psi(t):=U(t,s)\Psi_0$ solves Eq.~\eqref{eq:schro_weak} with initial condition $\Psi(0)=\Psi_0$ for all $\Phi\in\hilb^+$.
\end{definition}

For quantum control purposes, we often need to further relax our requests by taking into account propagators solving Eq.~\eqref{eq:schro_weak} for all but finitely many values of $t$, that is, admitting finitely many time singularities, see e.g.~\cite{boscain2012weak,boussaid2013weakly,chambrion2009controllability}. This leads us to the following definition:
\begin{definition}\label{def:admissible}
		Let $H(t)$, $t\in I$, be defined as above. We say that a unitary propagator $U(t,s)$, $t,s\in I$, is an \textit{admissible solution of the Schrödinger equation} if, for all $\Psi_0\in\hilb^+$, there exist $t_0<t_1<\dots<t_d\in I$ and a family of weak solutions of the Schrödinger equation $\{U_i(t,s) \mid t,s \in (t_{i-1},t_i)\}_{i=1,\dots,d}$ such that  for $t\in(t_{i-1},t_i)$, $s\in(t_{j-1},t_j)$, $1\leq j<i \leq d$ the unitary propagator $U(t,s)$ can be expressed as
		\begin{equation}
			U(t,s)=U_i(t,t_{i-1})U_{i-1}(t_{i-1},t_{i-2})\cdots U_j(t_{j},s).
		\end{equation}
	\end{definition}
Any strong solution of the Schrödinger equation is also, \textit{a fortiori}, a weak solution; however, the converse is not true. Indeed, $\Psi(t)$ solving Eq.~\eqref{eq:schro_weak} does not require the condition $U(t,s)\mathcal{D}(H(s))=\mathcal{D}(H(t))$, but merely the much weaker condition $U(t,s)\hilb^+=\hilb^+$. Besides, obviously a weak solution of the Schrödinger equation is also an admissible solution, and the converse is not true.

We shall now introduce the main setting of this work: a particle in a one-dimensional moving box. Consider a (non-relativistic) massive particle confined in a time-varying one-dimensional box with perfectly rigid walls. Without loss of generality, we will set $m=1/2$. In a fixed reference frame, the length and center position of the box will be described respectively by two real-valued functions $t\mapsto\ell(t),d(t)$, with $\ell(t)>0$. This means that the box corresponds to the interval
\begin{equation}
	\Omega_{\ell(t),d(t)}=\left[d(t)-\frac{\ell(t)}{2},d(t)+\frac{\ell(t)}{2}\right].
\end{equation}
The (time-dependent) Hilbert space of this setting is $L^2\bigl(\Omega_{\ell(t),d(t)}\bigr)$, and the Hamiltonian $H_{\ell(t),d(t)}$ associated with the system is the unbounded operator defined by
\begin{eqnarray}\label{eq:domham}
	\mathcal{D}\bigl(H_{\ell(t),d(t)}\bigr)&=&\mathcal{D}_{\rm Dir}\bigl(\Omega_{\ell(t),d(t)}\bigr) := \mathrm{H}^1_0\bigl(\Omega_{\ell(t),d(t)}\bigr)\cap\mathrm{H}^2\bigl(\Omega_{\ell(t),d(t)}\bigr);\\
	\label{eq:ham}
	H_{\ell(t),d(t)}&=&-\frac{\mathrm{d}^2}{\mathrm{d}x^2}.
\end{eqnarray}
At each instant of time the domain can be identified with the set of all functions in the second Sobolev space $\mathrm{H}^2(\Omega_{\ell(t),d(t)})$ that vanish at the boundary. 

In this scenario, the Hilbert space of the theory is time-dependent; strictly speaking, this makes the corresponding Schrödinger equation ill-defined. However, as discussed in~\cite{dimartino2013quantum}, this difficulty can be circumvented by embedding $L^2(\Omega_{\ell(t),d(t)})$ into the (time-independent) Hilbert space $L^2(\mathbb{R})\simeq L^2\bigl(\Omega_{\ell(t),d(t)}\bigr)\oplus L^2\bigl(\Omega^{\rm c}_{\ell(t),d(t)}\bigr)$, with $\Omega_{\ell(t),d(t)}^{\rm c}$ being the complement of $\Omega_{\ell(t),d(t)}$, and taking into account the Schrödinger equation generated by the Hamiltonian $H_{\ell(t),d(t)}\oplus I$, i.e. acting trivially on the orthogonal complement. With this choice, the dynamics on the orthogonal complement is considered to be trivial and plays no role in the discussion that follows: the controllability properties and results are determined completely by the projection onto the subspaces $L^2\bigl(\Omega_{\ell(t),d(t)}\bigr)$ and from now on we will consider only the restricted dynamics. We refer to~\cite{dimartino2013quantum} for further details.
	
With this caveat in mind, given a pair of functions $t\in \mathbb{R_+}\mapsto\ell(t),d(t)$, with $\ell(0)=\ell_0$ and $d(0)=d_0$, let us consider the control system~\refsigma{}  determined by the Schrödinger equation associated with the time-dependent Hamiltonian $H_{\ell(t),d(t)}$:
\begin{equation}\label{eq:schro}
	(\Sigma)\;
	\begin{cases}
		\ii\dot\Phi(t)=H_{\ell(t),d(t)}\Phi(t),& \Phi(t)\in \mathcal{D}_{\rm Dir}\bigl(\Omega_{\ell(t),d(t)}\bigr)\\
		\Phi(0)=\Phi_0,&\Phi_0\in\mathcal{D}_{\rm Dir}\left(\Omega_{\ell_0,d_0}\right).
	\end{cases}
\end{equation}
This is the equation governing the evolution of a quantum particle in the moving box $\Omega_{\ell(t),d(t)}$, assuming that at the initial time $t=0$ the wavefunction describing the particle is $\Phi_0$.

In this article we study the conditions under which the system~\refsigma{} admits a unique solution for the given initial condition, and also study its controllability. That is, we will determine the conditions for the existence of a unitary propagator (cf.\ Def.~\ref{def:unitary}) that solves Eq.~\eqref{eq:schro}, and for the existence of control functions $t\mapsto\ell(t),d(t)$ such that the system is approximately controllable. We refer to the next section for further details.

Following the construction in~\cite{dimartino2013quantum} it is possible to transform the system~\refsigma{} into a system with a fixed domain which is in the form~\eqref{eq:sum}.
As a starting point, for any $0<\ell\in\mathbb{R}$ and $d\in\mathbb{R}$, consider the unitary transformation
\begin{equation}\label{eq:w}
	W_{\ell,d}:L^2\left(\Omega_{\ell,d}\right)\rightarrow L^2\left(\Omega_{1,0}\right),\qquad \left(W_{\ell,d}\Phi\right)(x)=\sqrt{\ell}\,\Phi(\ell x+d).
\end{equation}
This is the unitary operator implementing a dilation $x\to x'=\ell x$ followed by a translation $x'\to x''=x'+d$. Hereafter, we shall set $\Omega\equiv\Omega_{1,0}$. In addition, one has
\begin{eqnarray}\label{eq:regularity}
	W_{\ell,d}\mathrm{H}^1\left(\Omega_{\ell,d}\right)=\mathrm{H}^1\left(\Omega\right),\qquad 	W_{\ell,d}\mathrm{H}^1_0\left(\Omega_{\ell,d}\right)=\mathrm{H}^1_0\left(\Omega\right),\qquad
	W_{\ell,d}\mathrm{H}^2\left(\Omega_{\ell,d}\right)=\mathrm{H}^2\left(\Omega\right).
\end{eqnarray}
By means of such a transformation it can be shown that the Schrödinger problem generated by the operator (with time-dependent domain) $H_{\ell(t),d(t)}$ can be solved by first solving the one generated by the auxiliary Hamiltonian $\tilde{H}_{\ell(t),d(t)}$ defined by
\begin{eqnarray}\label{eq:domtransf}
	\tilde{H}_{\ell(t),d(t)}&=&\frac{1}{\ell(t)^2}\Delta_{\rm Dir}-\frac{\dot\ell(t)}{\ell(t)}x\circ p-\frac{\dot d(t)}{\ell(t)}p
\end{eqnarray}
with time-independent domain $\mathcal{D}\bigl(\tilde{H}_{\ell(t),d(t)}\bigr) = \mathcal{D}_{\rm Dir}\bigl(\Omega\bigr)$, 
where $\Delta_{\rm Dir}$ is the (nonnegative) Dirichlet Laplacian on the unit interval, and $p$ and $x$ are the operators defined by	$\left(p\Psi\right)(x)=-\ii\Psi'(x)$, $\left(x\Psi\right)(x)=x\Psi(x)$ 
and $x\circ p:=\frac{1}{2}\left(xp+px\right)$. We can define now the auxiliary system~\reftildesigma{} determined by the auxiliary Hamiltonian:	
\begin{equation}\label{eq:schro2}
	(\tilde{\Sigma})\;
	\begin{cases}
		\ii\dot\Phi(t)=\tilde{H}_{\ell(t),d(t)}\Phi(t), &\Phi(t)\in \mathcal{D}_{\rm Dir}\bigl(\Omega\bigr)\\
		\Phi(0)=\Phi_0, &\Phi_0\in\mathcal{D}_{\rm Dir}\left(\Omega\right).
	\end{cases}
\end{equation}
	
The following result, which will be proven in Section~\ref{sec:prelimin}, holds as a consequence of some abstract results about Hamiltonians in the form of Eq.~\eqref{eq:sum}:
\begin{proposition}\label{prop:transform}
	Let $I\subset\mathbb{R}$ be a compact interval, and let $t\in I\mapsto\ell(t),d(t)\in\mathbb{R}$ two functions in $\mathrm{C}^2(I)$ with $\ell(t)>0$ for all $t\in I$. Then the time-dependent Hamiltonian $\tilde{H}_{\ell(t),d(t)}$, cf.\ Eqs.~\eqref{eq:domtransf}, is self-adjoint, semibounded from below uniformly in $t$, and the system~\reftildesigma{} has a unique weak solution $\tilde{U}(t,s)$, $t,s\in I$.
\end{proposition}
As a consequence:
\begin{corollary}\label{coroll:transform}
	Let $I\subset\mathbb{R}$, $t\in I\mapsto\ell(t),d(t)\in\mathbb{R}$ be as in Prop.~\ref{prop:transform}. Let $\tilde{U}(t,s)$, $t,s\in I$ be the unique weak solution of~\reftildesigma. Then, the unitary propagator defined by
	\begin{equation}\label{eq:unitaries}
		U(t,s):=W^\dag_{\ell(t),d(t)}\tilde{U}(t,s)W_{\ell(s),d(s)},\quad t,s\in I,
	\end{equation}
	is a weak solution of the system~\refsigma.
\end{corollary}
\begin{proof}
	The result follows from a direct comparison between the Schrödinger equations~\eqref{eq:schro} and~\eqref{eq:schro2} reported in~\cite{dimartino2013quantum}, and from Prop.~\ref{prop:transform}.
\end{proof}

\begin{remark}\label{rem:admissible}
	As an immediate consequence of Prop.~\ref{prop:transform} and its corollary, the following statement holds: given $t\in I\mapsto\ell(t),d(t)$ two functions in $\mathrm{C}^2_{\rm p}(I)$, both time-dependent Hamiltonians $\tilde{H}_{\ell(t),d(t)}$ and $H_{\ell(t),d(t)}$, $t\in I$ are self-adjoint, semibounded from below uniformly in $t$, and determine respectively unitary propagators $\tilde{U}(t,s)$, $U(t,s)$ which are \textit{admissible} solutions of their corresponding Schrödinger equations in the sense of Def.~\ref{def:admissible}.
\end{remark}

\subsection{Approximate controllability of a particle in a moving box} 

A \textit{quantum control system} is the dynamical system determined by the Schrödinger equation generated by a family of Hamiltonians $\{H_{c}\,|\,c\in\mathcal{C}\}$, with $\mathcal{C}$ being a suitable space of parameters (controls). A typical example is provided by a bilinear quantum control system, defined by a family of Hamiltonians $\{H_0+cV\,|\,c\in\mathcal{C}\subset\mathbb{R}\}$, where $H_0$ is a densely defined, self-adjoint operator on $\hilb$ with domain $\mathcal{D}(H_0)$, $V$ is a symmetric operator with domain $\mathcal{D}(V)\supseteq\mathcal{D}(H_0)$, and for all $c\in\mathcal{C}$ the operator defined by
\begin{equation}
	H_c=H_0+cV,\qquad\mathcal{D}(H_c)=\mathcal{D}(H_0),
\end{equation}
is self-adjoint on $\hilb$. In particular, if $V$ is infinitesimally form bounded with respect to $H_0$, cf.\ Def.~\ref{def:infform}, then the last point always holds as a straightforward consequence of the Kato--Lions--Lax--Milgram--Nelson (KLMN) theorem. $H_0$ is referred to as the \textit{drift Hamiltonian}, and it is responsible for the free or uncontrolled dynamics, while the operator $V$ represents an \textit{interaction} whose strength is modulated by the parameter. A typical example is a quantum particle subjected to an external field, e.g.\ an electric field. 

A system is approximately controllable if, by properly choosing a control function $f:[0,T]\mapsto\mathcal{C}$, the solution of the system generated by the time-dependent Hamiltonian $H_{f(t)}$ exists and drives the system from any initial state $\Phi_0$ arbitrarily close to any target state $\Phi_1$. The control function will depend on the initial and target state as well as on the desired precision.

Let us come back to the main setting of the work. Corollary~\ref{coroll:transform} ensures us that the time-dependent Hamiltonian associated with a quantum particle in a moving box determines a unitary propagator solving the Schrödinger equation for arbitrary motions of the walls, with the corresponding unitary propagator being expressed, cf.\ Eq.~\eqref{eq:unitaries}, by the unitary propagator of a problem with fixed domain. A question that arises naturally is whether it is possible to \textit{control} such a system. The affirmative result to this question is the main result of this work:

\begin{theorem}\label{thm:main}
	Let $\ell_0,\ell_1>0$, $d_0,d_1\in\mathbb{R}$, $r>0$ and $\Omega_{\ell_0,d_0}$, $\Omega_{\ell_1,d_1}\subset\mathbb{R}$ compact intervals. Take $\Delta d = d_1-d_0$ and $\Delta \ell = \ell_1-\ell_0$. Then:		
	\begin{itemize}
		\item[(i)] If $\Delta d\cdot\Delta\ell\neq0$ or $\Delta d=\Delta\ell=0$; for all $\epsilon>0$, $\Phi_0\in L^2(\Omega_{\ell_0,d_0})$, $\Phi_1\in L^2(\Omega_{\ell_1,d_1})$ with $\|\Phi_0\|=\|\Phi_1\|$, there exist $T>0$, $\delta\neq0$ and a piecewise linear function $f:[0,T]\rightarrow\mathbb{R}$, with $|f(t)|<\ell_0$, $|\dot f(t)|<r$ a.e., $f(0)=0$, $f(T)=\Delta \ell$, and such that the system~\refsigma{} with
		\begin{equation}\label{eq:motion1}
			\ell(t)=\ell_0 + f(t)\quad\text{and}\quad d(t) = d_0 +\delta f(t)
		\end{equation}
		has a unique admissible solution $U_f(t,s)$; $t,s\in[0,T]$ that satisfies
		\begin{equation}\label{eq:approxcontrol}
			\left\|\Phi_1-{U}_f(T,0)\Phi_0\right\|<\epsilon;
		\end{equation}
		\item[(ii)] If $\Delta d=0$; for all $\epsilon>0$, $\Phi_0\in L_{\pm}^2(\Omega_{\ell_0,d_0})$, $\Phi_1\in L^2_{\pm}(\Omega_{\ell_1,d_0})$ with $\|\Phi_0\|=\|\Phi_1\|$ and with the same parity, there exist $T>0$ and a piecewise linear function $f:[0,T]\rightarrow\mathbb{R}$, with $|f(t)|<\ell_0$, $|\dot f(t)|<r$ a.e., $f(0)=0$, $f(T)=\Delta \ell$, and such that the system~\refsigma{} with
		\begin{equation}\label{eq:motion2}
			\ell(t)=\ell_0 + f(t)\quad\text{and}\quad d(t) = d_0
		\end{equation}
		has a unique admissible solution $U_f(t,s)$; $t,s\in[0,T]$ that satisfies Eq.~\eqref{eq:approxcontrol}.
	\end{itemize}
\end{theorem}

The space $L^2_{+}(\Omega_{l,d})$ denotes the closed subspace of $L^2(\Omega_{l,d})$ with positive parity and equivalently for $L^2_{-}(\Omega_{l,d})$. Theorem~\ref{thm:main} covers two cases, see Fig.~\ref{fig:walls}.
\begin{itemize}
	\item We can drive the system between two arbitrary states confined in regions with different lengths and different center positions, without symmetry limitations. In order to achieve that, both $\ell(t)$ and $d(t)$ will be time-dependent and modulated by the control function $f(t)$ according to Eq.~\eqref{eq:motion1}. This includes the particular case of one fixed wall and the other one moving, a relevant and practical scenario.
	\item We can drive the system between two arbitrary states confined in regions with different lengths and same center position $d_0$, provided that a selection rule is satisfied: the two states must have definite, and equal, parity. In order to achieve that, only $\ell(t)$ will be time-dependent according to Eq.~\eqref{eq:motion2}, that is, the cavity must undergo a \textit{pure dilation} motion: the two walls must move symmetrically with respect to each other.
\end{itemize}
\begin{figure}[h!]
	\centering
	\begin{subfigure}[t]{0.48\textwidth}\centering	
		\includegraphics[width=\linewidth]{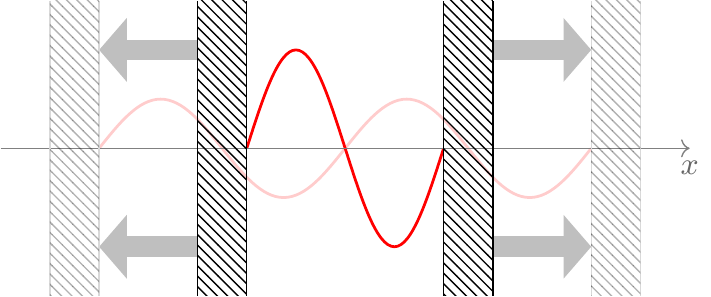}\caption{\centering Two moving walls, $d(t)=\text{const.}$ and \hbox{$\ell(t)\neq\text{const.}$}}\label{fig:walls_a}
	\end{subfigure}\hspace{0.05\linewidth}
	\begin{subfigure}[t]{0.41\textwidth}\centering
		\includegraphics[width=\linewidth]{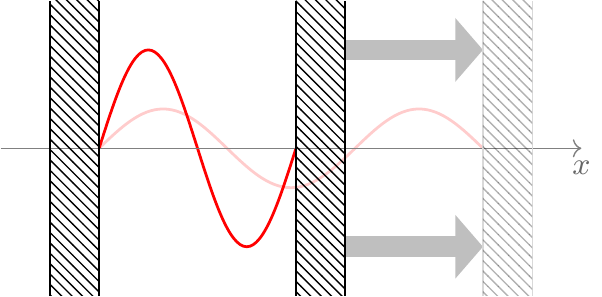}\caption{Single moving wall.}\label{fig:walls_b}
	\end{subfigure}
	\caption{Two particular cases covered by Theorem~\ref{thm:main}. (a) A purely dilating cavity, to which Theorem~\ref{thm:main}(ii) applies: approximate controllability between states with the same parity holds. (b) 
	A translating and dilating cavity, to which Theorem~\ref{thm:main}(i) applies: approximate controllability between arbitrary states holds.}
	\label{fig:walls}
\end{figure}

The proof of Theorem~\ref{thm:main}, to which the entirety of Section~\ref{sec:result} is devoted, proceeds in three steps:
\begin{itemize}
	\item[(a)] As a first step, an auxiliary control system related with the transformation $W_{\ell,d}$ in Eq.~\eqref{eq:w} is proven to be approximately controllable by piecewise constant functions, cf.\ Theorem~\ref{thm:auxiliary};
	\item[(b)] The stability results of Section~\ref{sec:prelimin} are used to prove approximate controllability with piecewise differentiable controls of the auxiliary control problem defined by the Hamiltonian $\tilde{H}_{\ell(t),d(t)}$ as defined in Eq.~\eqref{eq:domtransf};	
	\item[(c)] The time-dependent transformation $W_{\ell(t),d(t)}$ is used to prove approximate controllability of the control problem defined by the Dirichlet Laplacian in a box with moving walls.
\end{itemize}
In the next section we will introduce some mathematical preliminaries needed to prove the theorems.

\section{Existence of solutions and stability}\label{sec:prelimin}

In this section we introduce an important stability result, Theorem~\ref{thm:abstract2}, about the dynamics induced by families of time-dependent operators in the form
\begin{equation}\label{eq:sum2}
	H_n(t)=\sum_{i=0}^\nu f_{n,i}(t)H_i.
\end{equation}
with $\{H_0,H_1,\dots,H_\nu\}$ being a collection of symmetric operators on the Hilbert space $\hilb$, and with $\{f_{n,0}\}_{n\in\textbf{N}},...,\{f_{n,\nu}\}_{n\in\textbf{N}}$ sequences of measurable real-valued functions on a compact real interval $I$, where $\textbf{N}\subset\mathbb{N}$. Such results will be first applied to prove Prop.~\ref{prop:transform}, i.e.\ existence of weak solutions of the system~\reftildesigma{}, and then in Section~\ref{sec:result} to obtain the main result of this work.

Let us start with the following definition.
\begin{definition}\label{def:infform}
	Let $H_0,V$ be two symmetric operators on $\hilb$, with $\mathcal{D}(H_0)\subseteq\mathcal{D}(V)$. We say that $V$ is \textit{infinitesimally form bounded} with respect to $H_0$ if for all $\epsilon>0$ there exists $b(\epsilon)\geq0$ such that, for all $\Psi\in\mathcal{D}(H_0)$,
	\begin{equation}
		\left|\Braket{\Psi,V\Psi}\right|\leq\epsilon\Braket{\Psi,H_0\Psi}+b(\epsilon)\|\Psi\|^2.
	\end{equation}
\end{definition}
\begin{hypothesis}\label{hyp}
	We assume that:
	\begin{itemize}
		\item $H_0$ is a non-negative self-adjoint operator, with domain $\mathcal{D}(H_0)$ and form domain $\hilb^+$;
		\item $\{H_1,\dots,H_\nu\}$ is a collection of symmetric operators, with domain $\mathcal{D}(H_i)\supset\mathcal{D}(H_0)$, that are infinitesimally form bounded with respect to $H_0$.
	\end{itemize}
\end{hypothesis}
\begin{lemma}\label{lemma:kappa}
	Let $\{H_0,H_1,\dots,H_\nu\}$ be a collection of operators satisfying Hypothesis~\ref{hyp}. Define the norms
	\begin{equation}\label{eq:plusminus}
		\|\Psi\|_{\pm}=\left\|(H_0+1)^{\pm1/2}\Psi\right\|.
	\end{equation}
	Then there exists $K>0$ such that, for all $\Psi\in\hilb^+$,
	\begin{equation}\label{eq:kappa1}
		\left|\Braket{\Psi,H_i\Psi}\right|\leq K\|\Psi\|^2_+,\qquad i=0,1,\dots,\nu;
	\end{equation}
\end{lemma}
\begin{proof}
	Since for all $i=1,\dots,\nu$ $H_i$ is (infinitesimally) form bounded with respect to $H_0$, for all $\epsilon>0$,
	\begin{equation}
		\left|\Braket{\Psi,H_i\Psi}\right|\leq\epsilon\Braket{\Psi,H_0\Psi}+b_i(\epsilon)\|\Psi\|^2,
	\end{equation}
	and the desired inequality holds with
	\begin{equation}
		K=\max\{1,\epsilon,b_1(\epsilon),\dots,b_\nu(\epsilon)\}.\qedhere
	\end{equation}
\end{proof}

\begin{remark}\label{rem:scale}
	The form domain $\hilb^+$, endowed with the norm $\|\cdot\|_+$, is a Hilbert space strictly contained in $\hilb$. The completion of $\hilb$ with respect to the norm $\|\cdot\|_-$ will be denoted by $\hilb^-$; remarkably, the triplet of spaces $(\hilb^+,\hilb,\hilb^-)$ constitutes a scale of Hilbert spaces (or Gel'fand triple), that is,
	\begin{equation}
		\hilb^+\subset\hilb\subset\hilb^-,
	\end{equation}
	each inclusion being dense with respect to the topology of the larger space, and with the three norms satisfying $\|\cdot\|_-\leq\|\cdot\|\leq\|\cdot\|_+$. Among the many properties of such scales, we recall the following: for all $\Phi,\Psi\in\hilb^+$, a Cauchy--Schartz-like inequality holds:
	\begin{equation}\label{eq:cauchyschwartz}
		\left|\Braket{\Psi,\Phi}\right|\leq\|\Psi\|_{-}\|\Phi\|_+.
	\end{equation}
\end{remark}

\begin{lemma}\label{lemma:abstract}
	Let $\{H_0,H_1,\dots,H_\nu\}$ be a collection of symmetric operators satisfying Hypothesis~\ref{hyp}, and let $\{f_{n,0}\}_{n\in\textbf{N}},\dots,\{f_{n,\nu}\}_{n\in\textbf{N}}$ be sequences of measurable real-valued functions on a compact interval $I\subset\mathbb{R}$ such that 
	\begin{eqnarray}\label{eq:bounds0}
		M&:=&\sup\left\{|f_{n,i}(t)|\,|\,t\in I,\:i=0,\dots,\nu,\: n\in\textbf{N}\right\}<\infty,\nonumber\\\mu&:=&\inf\left\{f_{n,0}(t)\,|\:t\in I,\:n\in\textbf{N}\right\}>0.
	\end{eqnarray}
	Define $H_n(t)$ as in Eq.~\eqref{eq:sum2}. Then there exist $\mathcal{D}\left(H_n(t)\right)\subset\hilb$ such that $\{H_n(t)\}_{n\in\mathbf{N}}$ is a family of time-dependent Hamiltonians, with domains $\mathcal{D}\left(H_n(t)\right)$,  self-adjoint, semibounded from below uniformly in $t$ and $n$, and having form domain $\hilb^+$.
\end{lemma}
\begin{proof}
	Let $\epsilon>0$. By assumption, there exist $b_1(\epsilon),\dots,b_\nu(\epsilon)$ such that, for all $i=1,\dots,\nu$ and $\Psi\in\mathcal{D}(H_0)$,
	\begin{equation}
		\left|\Braket{\Psi,H_i\Psi}\right|\leq\frac{\epsilon\mu}{\nu M}\Braket{\Psi,H_0\Psi}+b_i(\epsilon)\|\Psi\|^2,
	\end{equation}
	therefore
	\begin{eqnarray}\label{eq:formb}
		\left|\Braket{\Psi,\sum_{i=1}^\nu f_{n,i}(t)H_i\Psi}\right|&\leq&M\sum_{i=1}^\nu\left|\Braket{\Psi,H_i\Psi}\right|
		\leq\epsilon\mu\Braket{\Psi,H_0\Psi}+Mb(\epsilon)\|\Psi\|^2,
	\end{eqnarray}
	with $b(\epsilon)=\sum_ib_i(\epsilon)$. Therefore, the operator $\sum_{i=1}^\nu f_{n,i}(t)H_i$, defined on $\mathcal{D}(H_0)$, is infinitesimally form bounded with respect to $f_{n,0}(t)H_0$, the latter being non-negative by assumption. By the KLMN theorem, $H_n(t)$ will thus admit a self-adjointness domain $\mathcal{D}(H_n(t))$.
	Besides,
	\begin{eqnarray}
		\Braket{\Psi,H_n(t)\Psi}&\geq&\left(f_{n,0}(t)-\mu\epsilon\right)\Braket{\Psi,H_0\Psi}-Mb(\epsilon)\|\Psi\|^2\nonumber\\
		&\geq&\mu(1-\epsilon)\Braket{\Psi,H_0\Psi}-Mb(\epsilon)\|\Psi\|^2.
	\end{eqnarray}
	This holds for all $\epsilon$. In particular, choosing any $\epsilon<1$, we have
	\begin{equation}
		\Braket{\Psi,H_n(t)\Psi}\geq-Mb(\epsilon)\|\Psi\|^2,
	\end{equation}
	hence $H_n(t)$ is semibounded from below uniformly in $n\in\textbf{N}$.
\end{proof}
The following theorem ensures that, under Hypothesis~\ref{hyp}, a family of Hamiltonians in the form~\eqref{eq:sum2}, with the sequences of coefficients satisfying suitable bounds, yields a family of unitary propagators $U_n(t,s)$ each solving the corresponding Schrödinger equation; in addition, this family is \textit{stable}.
\begin{theorem}\label{thm:abstract2}
	Let $\{H_0,H_1,\dots,H_\nu\}$ be a collection of symmetric operators satisfying Hypothesis~\ref{hyp}, $\textbf{N}\subset\mathbb{N}$, and let \hbox{$\{f_{n,0}\}_{n\in\textbf{N}}$, ..., $\{f_{n,\nu}\}_{n\in\textbf{N}}$} be sequences of measurable real-valued functions on a compact interval $I$ such that
	\begin{eqnarray}\label{eq:bounds1}
		M&:=&\sup\left\{|f_{n,i}(t)|\,\Big|\,t\in I,\:i=0,\dots,\nu,\: n\in\textbf{N}\right\}<\infty,\nonumber\\\mu&:=&\inf\left\{f_{n,0}(t)\,|\:t\in I,\:n\in\textbf{N}\right\}>0.
	\end{eqnarray}
	Then, for all $n\in\textbf{N}$, there exists $\mathcal{D}\left(H_n(t)\right)$ such that the time-dependent operator of Eq.~\eqref{eq:sum2}, with domain $\mathcal{D}\left(H_n(t)\right)$, satisfies the following properties:
	\begin{itemize}
		\item [(a)] Define the norms $\|\cdot\|_{\pm,n,t}$ on $\hilb^{\pm}$ via
		\begin{equation}
			\left\|\Psi\right\|_{\pm,n,t}:=\left\|(H_n(t)+m+1)^{\pm1/2}\Psi\right\|,
		\end{equation}
		with $m$ the uniform lower bound of the operators $H_n(t)$. Then there exists $c\geq1$ such that, for all $n\in\textbf{N}$, $t\in I$, and $\Psi\in\hilb^\pm$,
		\begin{equation}\label{eq:cequiv1}
			c^{-1}\left\|\Psi\right\|_{\pm,n,t}\leq\|\Psi\|_\pm\leq c\left\|\Psi\right\|_{\pm,n,t},
		\end{equation}
		with $\|\cdot\|_\pm$ as in Eq.~\eqref{eq:plusminus}.
		\item [(b)]If, in addition, $\{f_{n,0}, \dots, f_{n,\nu}\}_{n\in\textbf{N}}\subset\mathrm{C}^1(I)$ and
		\begin{equation}\label{eq:bounds1bis}
			\sup\left\{\|\dot{f}_{n,i}\|_{L^1(I)}\,\Big|\,i=0,\dots,\nu,\;\,n\in\textbf{N}\,\right\}<\infty,
		\end{equation}
		then for all $n\in\textbf{N}$ there exists a weak solution $U_n(t,s)$, $t,s\in I$, of the Schrödinger equation generated by $H_n(t)$, cf.\ Definition~\ref{def:weaksolution};
		\item[(c)] If, in addition,  $\{f_{n,0}, \dots, f_{n,\nu}\}_{n\in\textbf{N}}\subset\mathrm{C}^2(I)$ then, for all $n,m\in\textbf{N}$, $t,s\in I$, and $\Psi\in\hilb^+$, we have
		\begin{equation}\label{eq:stability1}
				\left\|\left(U_n(t,s)-U_m(t,s)\right)\Psi\right\|_-\leq L\|\Psi\|_+\sum_{i=0}^\nu\left\|f_{n,i}-f_{m,i}\right\|_{L^1(t,s)},
		\end{equation}
		where
		\begin{equation}\label{eq:l}
			L=c^8\exp\left(
			2c^2K(\nu+1)
			\,\sup\left\{\bigl\|\dot f_{n,i}\bigr\|_{L^1(I)}\,\Big|\,n\in\textbf{N},\:0\leq i\leq \nu\right\}\right),
		\end{equation}
		with $K$ as in Eq.~\eqref{eq:kappa1}, and $c$ as in Eq.~\eqref{eq:cequiv1}.
	\end{itemize}
\end{theorem}
\begin{proof}
	(a) The inequality~\eqref{eq:formb} implies, for all $t\in I$ and $n\in\textbf{N}$,
	\begin{equation}
		\left(f_{n,0}(t)-\mu\epsilon\right)\Braket{\Psi,H_0\Psi}-Mb(\epsilon)\|\Psi\|^2\leq\Braket{\Psi,H(t)\Psi}\leq\left(f_{n,0}(t)+\mu\epsilon\right)\Braket{\Psi,H_0\Psi}+Mb(\epsilon)\|\Psi\|^2,
	\end{equation}
	thus, since $\mu<f_{n,0}(t)\leq M$,
	\begin{equation}
		\mu(1-\epsilon)\Braket{\Psi,H_0\Psi}-Mb(\epsilon)\|\Psi\|^2\leq\Braket{\Psi,H(t)\Psi}\leq\left(M+\mu\epsilon\right)\Braket{\Psi,H_0\Psi}+Mb(\epsilon)\|\Psi\|^2.
	\end{equation}
	Fix $\epsilon<1$, define $m:=Mb(\epsilon)$, and add $(m+1)\|\Psi\|^2$ to each term of the inequality above. We get
	\begin{equation}
		\mu(1-\epsilon)\Braket{\Psi,H_0\Psi}+\|\Psi\|^2\leq\|\Psi\|^2_{+,n,t}\leq\left(M+\mu\epsilon\right)\Braket{\Psi,H_0\Psi}+(1+2m)\|\Psi\|^2,
	\end{equation}
	where $\|\Psi\|^2_{+,n,t}=\Braket{\Psi,(H_n(t)+m+1)\Psi}$. Then, defining
	\begin{equation}
		c=\max\left\{M+\mu\epsilon,1+2m,\frac{1}{\mu(1-\epsilon)},1\right\},
	\end{equation}
	we have
	\begin{equation}
		c^{-1}\|\Psi\|^2_+\leq\|\Psi\|^2_{+,n,t}\leq c\|\Psi\|^2_+,
	\end{equation}
	hence the desired equivalence between the norms $\|\cdot\|_+$ and $\|\cdot\|_{+,n,t}$. Recalling that $(\hilb^+,\hilb,\hilb^-)$ is a scale of Hilbert spaces (cf.\ Remark~\ref{rem:scale}), the equivalence between $\|\cdot\|_-$ and $\|\cdot\|_{-,n,t}$ follows consequently, cf.~\cite[Theorem 3]{balmaseda2021schrodinger}.
	
	Lemma~\ref{lemma:kappa}, Lemma~\ref{lemma:abstract}, and the bounds~\eqref{eq:bounds1}--\eqref{eq:bounds1bis} imply the assumptions in~\cite[Theorem 8.1]{Kisynski1964}, from which \textit{(b)} follows, and also those in~\cite[Proposition 3.10]{balmaseda2021controllability} and~\cite[Theorem 3.7]{balmaseda2021controllability} from which \textit{(c)} follows.
\end{proof}
We can now prove Prop.~\ref{prop:transform}, thus showing that the Schrödinger equation generated by the
time-dependent Hamiltonian of a particle in a moving box, system~\refsigma, has admissible solutions, see Remark~\ref{rem:admissible}.

\begin{lemma}\label{lemma:relbounded}
	The operators $p$ and $x\circ p$, defined on $\mathcal{D}_{\rm Dir}(\Omega)$, are infinitesimally form bounded with respect to $\Delta_{\rm Dir}$.
\end{lemma}
\begin{proof}
	Indeed, noticing that $\|\Psi'\|^2=\Braket{\Psi,\Delta_{\rm Dir}\Psi}$ for all $\Psi\in\mathcal{D}_{\rm Dir}(\Omega)$, and using Young's inequality with $\epsilon$, we get
	\begin{equation}
		\left|\braket{\Psi,p\Psi}\right|\leq\epsilon\Braket{\Psi,\Delta_{\rm Dir}\Psi}+\frac{1}{4\epsilon}\|\Psi\|^2.
	\end{equation}
	The case for $x\circ p$ is proven similarly having into account that $x$ is a bounded operator.
\end{proof}

\begin{proof}[Proof of Prop.~\ref{prop:transform}]
	The operator $\Delta_{\rm Dir}$, i.e.\ the Dirichlet Laplacian with domain $\mathcal{D}(\Delta_{\rm Dir})=\mathcal{D}_{\rm Dir}(\Omega)$, is self-adjoint and non-negative. Since $p$ and $x\circ p$ are infinitesimally form bounded with respect to $\Delta_{\rm Dir}$, by the KLMN theorem the operators satisfy Hyp.~\ref{hyp}. Besides, since $t\mapsto\ell(t),d(t)$ are in $\mathrm{C}^2(I)$ by assumption (and thus $t\mapsto\dot{d}(t),\dot\ell(t)$ are in $\mathrm{C}^1(I)$), and $\ell(t)>0$, the functions
	\begin{equation}
		t\in I\mapsto\frac{1}{\ell(t)^2},\;\frac{\dot{d}(t)}{\ell(t)},\;\frac{\dot{\ell}(t)}{\ell(t)}
	\end{equation}
	are all in $\mathrm{C}^1(I)$. The bounds of Eq.~\eqref{eq:bounds1} are then satisfied. Therefore, Theorem~\ref{thm:abstract2} applies, and Part \textit{(b)} provides the result.
\end{proof}
We stress that, as pointed out in~\cite{dimartino2013quantum} as well, a similar result would be obtained replacing Dirichlet boundary conditions with other boundary conditions that are dilation-invariant, such as Neumann boundary conditions. Besides, the fact that such bounds are infinitesimal is crucial: this guarantees that self-adjointness and semiboundedness hold for arbitrary values of $\ell(t)$, $d(t)$, and their derivatives, as opposed to what would happen if they were relatively bounded with strictly positive relative bound.

\section{Proving controllability by moving walls}\label{sec:result}

The first step towards the proof of Theorem~\ref{thm:main} will be to reduce the degrees of freedom of the motion of the box. Let $\delta\in\mathbb{R},\lambda\geq0$, and consider two functions $t\mapsto d(t),\ell(t)$ in the form
\begin{eqnarray}
	d(t)&=&d_0+\delta\,f(t);\label{eq:dft}\\
	\ell(t)&=&\ell_0+\lambda\,f(t),\label{eq:lft}
\end{eqnarray}
for some measurable real-valued function $t\mapsto f(t)$, satisfying $f(0)=0$ and (if $\lambda>0$) $|f(t)|<\ell_0/\lambda$. This entails considering a box whose walls lie at positions $x_-(t)<x_+(t)$, $t\in\mathbb{R}$, given by
\begin{eqnarray}
	x_-(t)&=&\left(d_0-\frac{\ell_0}{2}\right)+\left(\delta-\frac{\lambda}{2}\right)f(t);\\
	x_+(t)&=&\left(d_0+\frac{\ell_0}{2}\right)+\left(\delta+\frac{\lambda}{2}\right)f(t).
\end{eqnarray}
Different values of the parameters $\lambda,\delta$ correspond (up to a global multiplicative constant, which can be reabsorbed into $f(t)$) to different kinds of motions of the walls, including:
\begin{itemize}
	\item A purely dilating box whose center lies at the fixed position $d(t)=d_0$ for $\delta=0$ (Fig.~\ref{fig:walls_a}).
	\item A box with a single moving wall for $\delta=\pm \lambda$ (Fig.~\ref{fig:walls_b}).
 	\item A purely translating box with fixed length $\ell(t)=\ell_0$ for $\lambda=0$.
\end{itemize}
The particular cases of the Hamiltonians $H_{\ell(t),d(t)}$ and $\tilde{H}_{\ell(t),d(t)}$ for $d(t),\ell(t)$ given by Eqs.~\eqref{eq:dft}--\eqref{eq:lft} will be denoted by $H^{\lambda,\delta}_{f(t)}$ and $\tilde{H}^{\lambda,\delta}_{f(t)}$ respectively. From Eq.~\eqref{eq:domtransf} we get
\begin{equation}\label{eq:transf2}
	\tilde{H}^{\lambda,\delta}_{f(t)}=\frac{1}{\left[\ell_0+\lambda f(t)\right]^2}\Delta_{\rm Dir}-\frac{\dot f(t)}{\ell_0+\lambda f(t)}V^{\lambda,\delta},
\end{equation}
where
\begin{equation}
	V^{\lambda,\delta}=\lambda\,x\circ p+\delta\,p.
\end{equation}
Their corresponding unitary propagators will be denoted by $U^{\lambda,\delta}_f(t,s)$ and $\tilde{U}^{\lambda,\delta}_f(t,s)$. Notice that, as a direct consequence of Corollary~\ref{coroll:transform} and Eqs.~\eqref{eq:dft}--\eqref{eq:lft}, they are related by
\begin{equation}\label{eq:unitaries2}
	U^{\lambda,\delta}_f(t,s)=W^\dag_{\ell_0+\lambda f(t),d_0+\delta f(t)}\tilde{U}^{\lambda,\delta}_f(t,s)W_{\ell_0+\lambda f(s),d_0+\delta f(s)}.
\end{equation}
The remainder of this section is devoted to proving Theorem~\ref{thm:main}. We will rely on a result for bilinear quantum control systems, given in~\cite{chambrion2009controllability,boscain2012weak}. To this purpose, there are two main technical obstacles to overcome, namely:
\begin{enumerate}
	\item The Hamiltonian $\tilde{H}^{\lambda,\delta}_{f(t)}$ does not define a bilinear control system, because of the time-dependence of the drift term and the dependence on both $f(t)$ and its derivative $\dot f(t)$;
	\item The quantum control system~\reftildesigma{} defined by the Hamiltonians $\tilde{H}^{\lambda,\delta}_{f(t)}$ is equivalent to the quantum control system~\refsigma{} only if
	\begin{itemize}
		\item $f$ is sufficiently regular so that Corollary~\ref{coroll:transform} can be applied.
		\item the values of the control function $f(t)$ at the initial and final time are known beforehand.
	\end{itemize}
\end{enumerate}
Both obstacles will be overcome via a stability argument based on Theorem~\ref{thm:abstract2}. In order to apply the result of controllability on bilinear quantum control systems, we will need to verify for our auxiliary control problem the existence of a non-resonant connectedness chain. This is defined next.
\begin{definition}\label{def:chain}
	Let $H_0$ be a self-adjoint operator on $\hilb$ with compact resolvent, its spectrum being $\{E_j\}_{j\in\mathbb{N}}$, and with $\{\varphi_j\}_{j\in\mathbb{N}}\subset\hilb$ being a complete orthonormal set of eigenvectors of $H_0$. Let $S\subset\mathbb{N}^2$. Then $S$ is a \textit{connectedness chain} for the quantum control system $H_0+uV,\,u\in(-r,r)$ if the following conditions hold: for all $(j,\ell)\in\mathbb{N}^2$, there is a finite sequence
	\begin{equation}
		(j,r_1),\;(r_1,r_2),\;(r_2,r_3),\;\ldots,\;(r_{k-1},r_k),\;(r_{k},\ell)\in S
	\end{equation}
	such that
	\begin{equation}
		\Braket{\varphi_j,V\varphi_{r_1}},\;\Braket{\varphi_{r_1},V\varphi_{r_2}},\;\ldots,\;\Braket{\varphi_{r_{k-1}},V\varphi_{r_k}},\;\Braket{\varphi_{r_{k}},V\varphi_{\ell}}\neq0.
	\end{equation}
	Furthermore, $S$ is \textit{non-resonant} if, for all $(s_1,s_2)\in S$ and for all $(t_1,t_2)\in\mathbb{N}^2$ such that $\Braket{\varphi_{t_1},V\varphi_{t_2}}\neq0$, we have
	\begin{equation}
		|E_{s_2}-E_{s_1}|\neq|E_{t_2}-E_{t_1}|,
	\end{equation}
	excluding the trivial cases $(s_1,s_2)=(t_1,t_2)$ and $(s_1,s_2)=(t_2,t_1)$.
\end{definition}
As a particular case, suppose that the following condition holds:
\begin{equation}
	\Braket{\varphi_{j+1},V\varphi_j}\neq0\qquad\forall j\in\mathbb{N}.
\end{equation}
In such a case, the system clearly admits a connectedness chain $S$ given by the family of all couples of consecutive integers,
\begin{equation}\label{eq:consecutive}
	S=\{(s,s+1)\,|\;s\in\mathbb{N}\}\cup\{(s+1,s)\,|\;s\in\mathbb{N}\}.
\end{equation}

\subsection{Step one: an auxiliary control system}

We introduce an auxiliary control system to which known results by Boscain et al., cf.~\cite[Theorem 2.6]{boscain2012weak}, can be applied. Given any measurable function $v(t)$, define
\begin{eqnarray}\label{eq:auxiliary}
	\hat{H}^{\lambda,\delta}_{v}(t)&=&\frac{1}{\ell_0^2}\Delta_{\rm Dir}-\frac{1}{\ell_0}v(t)\,V^{\lambda,\delta},
\end{eqnarray}
which is formally obtained by replacing $\dot f(t)$ with $v(t)$ and setting $f(t)\equiv0$ in $\tilde{H}^{\lambda,\delta}_{f(t)}$, cf.\ Eq.~\eqref{eq:transf2}. This expression defines a bilinear quantum control system. We will show that such a system is approximately controllable for almost all values of $\delta$ and $\lambda$:
\begin{theorem}\label{thm:auxiliary}
	Let $\ell_0>0$, $r>0$, $\delta\in\mathbb{R}$, and $\lambda>0$. Then:
	\begin{itemize}
		\item[(i)] if $\delta\neq0$, the bilinear quantum control system defined by Eq.~\eqref{eq:auxiliary} is approximately controllable in $L^2(\Omega)$ via piecewise constant functions $v:[0,T]\rightarrow(-r,r)$;
		\item[(ii)] if $\delta=0$, defining
		\begin{equation}
			L^2_\pm(\Omega)=\left\{\Psi\in L^2(\Omega)\,|\;\Psi(-x)=\pm\Psi(x)\;\text{a.e.}\right\},
		\end{equation}
		the bilinear quantum control system defined by Eq.~\eqref{eq:auxiliary} is approximately controllable in $L^2_\pm(\Omega)$ via piecewise constant functions $v:[0,T]\rightarrow(-r,r)$.
	\end{itemize}
\end{theorem}
In other words, in the case $\lambda>0$ and $\delta\neq0$ (translation+dilation) the system is approximately controllable; in the case $\lambda>0$ and $\delta=0$ (pure dilation) the system is still approximately controllable up to a selection rule: given two states $\Psi_0,\Psi_1$ \textit{with the same, definite parity}, we can always drive each of them arbitrarily close to the other one. States with distinct or indefinite parity cannot be driven close to each other. The reason for that is the following: all terms in the Hamiltonian leave each parity sector invariant, and thus the evolution induced by it cannot couple the two sectors, no matter how the control function is chosen. The case $\lambda=0$, which corresponds to a pure translation, remains an open problem: the method developed in this work cannot be used to prove controllability in this case.

The rest of this subsection will be devoted to prove Theorem~\ref{thm:auxiliary}. Before that, notice that the value of $\ell_0$ is completely immaterial (as long as it is positive), since we can always rescale the Hamiltonian in such a way that its value only affects the definition of the control function $v(t)$. Consequently, in the remainder of this subsection, we will set $\ell_0=1$ without loss of generality and study the bilinear control system with drift Hamiltonian $\Delta_{\rm Dir}$ and interaction determined by $V^{\lambda,\delta}$.

The Hamiltonian $\Delta_{\rm Dir}$ has compact resolvent and a purely discrete, simple spectrum $\{E_j\}_{j\geq1}$, with corresponding normalized eigenvectors $\{\varphi_j\}_{j\geq1}$, given by
\begin{equation}\label{eq:h0eig}
	E_j=j^2\pi^2,\qquad\varphi_j(x)=\sqrt{2}\sin\left[j\pi\left(x+\frac{1}{2}\right)\right].
\end{equation}
Approximate controllability via piecewise constant functions is thus guaranteed by~\cite[Theorem 2.6]{boscain2012weak} provided that a non-resonant connectedness chain $S\subset\mathbb{N}^2$, cf.\ Def.~\ref{def:chain}, can be found. In our case, an immediate computation shows the following result. For all $j\neq\ell$,
\begin{eqnarray}\label{eq:brakets2}
	\Braket{\varphi_{j},p\varphi_\ell}&=&\frac{2j\ell}{\ell^2-j^2}\bigl[1-(-1)^{j+\ell}\bigr]\,(1-\delta_{j\ell}),\\
	\label{eq:brakets1}
	\Braket{\varphi_{j},\left(xp+px\right)\varphi_\ell}&=&\frac{2j\ell}{\ell^2-j^2}\bigl[1+(-1)^{j+\ell}\bigr]\,(1-\delta_{j\ell}),
\end{eqnarray}
that is, the operators $p$ and $x\circ p$ couple respectively eigenvectors whose indices differ by an odd and a non-zero even number. In particular, we have the following special cases:
\begin{itemize}
	\item $V^{0,1}=p$ only couples eigenvectors with different parity; correspondingly, the control system admits a connectedness chain;
	\item $V^{1,0}=x\circ p$ only couples eigenvectors with the same parity; correspondingly, the control system does not admit any connectedness chain.
\end{itemize}
Consequently, for all $\lambda>0$ and $0\neq \delta\in\mathbb{R}$, the control system determined by the Hamiltonian with interaction term $V^{\lambda,\delta}$ admits a connectedness chain $S$. In all those cases, a possible (but not unique) choice for $S$ is the set of all couples of consecutive integers as in Eq.~\eqref{eq:consecutive}. For example, in all cases in which $\lambda$ and $\delta$ are both non-zero, $S=\mathbb{N}^2$ itself is a possible choice for a connectedness chain, since all couples of eigenvectors are directly coupled by $V^{\lambda,\delta}$.

We need to choose a \textit{non-resonant} connectedness chain. However, the control system \textit{does} exhibit resonances. Indeed, given two couples of integers $(s_1,s_2)\in S,\,(t_1,t_2)\in\mathbb{N}^2$, a resonance occurs when
\begin{equation}
	\left|s_2^2-s_1^2\right|=\left|t_2^2-t_1^2\right|.
\end{equation}
Even by choosing the chain $S$ as in Eq.~\eqref{eq:consecutive}, thus constraining $(s_1,s_2)$ to be a couple of consecutive integers, resonances still occur, an example being
\begin{equation}
	(s_1,s_2)=(220,221),\qquad (t_1,t_2)=(20,29).
\end{equation}
This prevents us from directly applying the results in~\cite[Theorem 2.6]{boscain2012weak}; nevertheless, we will show that approximate controllability for the control system does hold. The analyticity of the eigenvalues and eigenvectors as functions of $\lambda$ will allow us to circumvent the resonance problem. This is the content of the following proposition.
\begin{proposition}\label{rem:controlperturb}
	The bilinear quantum control system $H_0+uV$, $u\in(-r,r)$ is approximately controllable if and only if, given an arbitrary $\eta\in(-r,r)$, the bilinear quantum control system $H_0+(\eta+u)V$ with $u\in(-r-\eta,r-\eta)$ is approximately controllable as well.
\end{proposition}
\begin{proof}
	Given $T>0$ and any function $t\in[0,T]\mapsto f(t)\in(-r,r)$, we can write
	\begin{equation}
		H_0+f(t)V=H_0+\eta V+\left(f(t)-\eta\right)V,
	\end{equation}
	whence the claim follows immediately.
\end{proof}
Prop.~\ref{rem:controlperturb} allows us to circumvent the problem of resonances in the following way:
\begin{lemma}\label{lemma:nonresonant}
	Let $\lambda > 0$, $\delta\neq 0$, and let $S\subset\mathbb{N}^2$ be a connectedness chain for the system $\Delta_{\rm Dir}+uV^{\lambda,\delta},\,u\in(-r,r)$. Then, for all $\eta_0>0$, there exists $0<\eta<\eta_0$ such that $S$ is a non-resonant connectedness chain for the control system~$\Delta_{\rm Dir}+(\eta+u)V^{\lambda,\delta}$ with $u\in(-\eta,r-\eta)$.
\end{lemma}
\begin{proof}
	Define, for $\eta\in\mathbb{C}$,
	\begin{equation}\label{eq:hhat}
		\hat{H}^{\lambda,\delta}(\eta)=\Delta_{\rm Dir}+\eta V^{\lambda,\delta}.
	\end{equation}
	By the relative boundedness of $V^{\lambda,\delta}$ with respect to $\Delta_{\rm Dir}$, cf.\ Lemma~\ref{lemma:relbounded}, it follows that this is a self-adjoint holomorphic family of type A, with $\hat{H}^{\lambda,\delta}(0)=\Delta_{\rm Dir}$ having compact resolvent. Consequently, by~\cite[Chapter VII, Theorem 2.4]{kato2013perturbation} and~\cite[Chapter VII, Theorem 3.9]{kato2013perturbation}, $\hat{H}^{\lambda,\delta}(\eta)$ has compact resolvent for all $\eta\in\mathbb{C}$, and there exist real analytic functions
	\begin{equation}
		E_{j}^{\lambda,\delta}:\mathbb{R}\rightarrow\mathbb{R},\qquad\varphi_{j}^{\lambda,\delta}:\mathbb{R}\rightarrow\hilb
	\end{equation}
	such that, for all $\eta\in\mathbb{R}$, $\{E_{j}^{\lambda,\delta}(\eta)\}_{j\geq1}$ is the spectrum of $H^{\lambda,\delta}(\eta)$, and $\{\varphi_{j}^{\lambda,\delta}(\eta)\}_{j\geq1}$ is its associated complete orthonormal family of eigenvectors, with $E_j^{\lambda,\delta}(0)=E_j$, $\varphi_j^{\lambda,\delta}(0)=\varphi_j$ as defined in Eq.~\eqref{eq:h0eig}. The first and second derivatives of both functions at $\eta=0$ can be computed using perturbation theory, see e.g.~\cite{kato2013perturbation}:
	\begin{eqnarray}
		\frac{\mathrm{d}}{\mathrm{d}\eta}E^{\lambda,\delta}_j(\eta)\bigg|_{\eta=0}=\Braket{\varphi_j,V^{\lambda,\delta}\varphi_j},\qquad
		\frac{\mathrm{d}^2}{\mathrm{d}\eta^2}E^{\lambda,\delta}_j(\eta)\bigg|_{\eta=0}=\sum_{\ell\neq j}\frac{\left|\Braket{\varphi_\ell,V^{\lambda,\delta}\varphi_j}\right|^2}{E_j-E_\ell}.
	\end{eqnarray}
	By Eqs.~\eqref{eq:brakets2}--\eqref{eq:brakets1}, the first derivative is thus zero, while the second one can be evaluated explicitly:
	\begin{equation}\label{eq:secondder}
		\frac{\mathrm{d}^2}{\mathrm{d}\eta^2}E^{\lambda,\delta}_j(\eta)\bigg|_{\eta=0}=\frac{\lambda^2}{8j^2\pi^2}-\frac{\lambda^2}{48}-\frac{\delta^2}{4}.
	\end{equation}
	Fix $(s_1,s_2)\in S\subset\mathbb{N}^2$ and $(t_1,t_2)\in\mathbb{N}^2$. We shall assume $s_2>s_1$ and $t_2>t_1$ without loss of generality, since in other cases we can relabel the indices accordingly. A resonance at $\eta=0$ occurs if and only if the real analytic function
	\begin{equation}\label{eq:resonanceta}
		\eta\mapsto f^{\lambda,\delta}_{s_1,s_2,t_1,t_2}(\eta)= \left(E_{s_2}^{\lambda,\delta}(\eta)-E_{s_1}^{\lambda,\delta}(\eta)\right)-\left(E_{t_2}^{\lambda,\delta}(\eta)-E_{t_1}^{\lambda,\delta}(\eta)\right)
	\end{equation}
	has a zero at $\eta=0$. Let $\Upsilon_{s_1,s_2,t_1,t_2}$ be the set of all values of $\eta$ for which the function above has a zero. By analyticity, either $\Upsilon_{s_1,s_2,t_1,t_2}$ is countable or it is full, i.e.\ the function above is identically zero. We shall show that the latter case never occurs.
	
	There are two cases. If the eigenvalues $E_{s_1},E_{s_2},E_{t_1},E_{t_2}$ are not resonant, then the function~\eqref{eq:resonanceta} is not identically zero. Let us then suppose that these eigenvalues are resonant, which implies
	\begin{equation}\label{eq:zerothorder}
		s_2^2-s_1^2=t_2^2-t_1^2.
	\end{equation}
	and therefore the function~\eqref{eq:resonanceta} has a zero at $\eta=0$. In order for it to be identically zero, its second derivative at $\eta=0$ should vanish as well. By Eq.~\eqref{eq:secondder}, the latter reads
	\begin{eqnarray}
		\frac{\mathrm{d}^2}{\mathrm{d}\eta^2}f^{\lambda,\delta}_{s_1,s_2,t_1,t_2}(\eta)\bigg|_{\eta=0}&=&-\frac{\lambda^2}{8\pi^2}\left[\left(\frac{1}{s_1^2}-\frac{1}{s_2^2}\right)-\left(\frac{1}{t_1^2}-\frac{1}{t_2^2}\right)\right],
	\end{eqnarray}
	thus it vanishes if and only if
	\begin{equation}
		\frac{s_2^2-s_1^2}{s_1^2s_2^2}=\frac{t_2^2-t_1^2}{t_1^2t_2^2}.
	\end{equation}
	Since Eq.~\eqref{eq:zerothorder} holds, the latter condition holds if and only if $s_1s_2=t_1t_2$. Therefore, in order for the function~\eqref{eq:resonanceta} to vanish identically, we must have necessarily
	\begin{equation}
		s_2^2-s_1^2=t_2^2-t_1^2,\qquad
		s_1s_2=t_1t_2,
	\end{equation}
	or equivalently
	\begin{equation}
		s_2^2+t_1^2=t_2^2+s_1^2,\qquad
		\frac{s_2}{t_1}=\frac{t_2}{s_1},
	\end{equation}
	which, recalling that $s_1,s_2,t_1,t_2\in\mathbb{N}$, is equivalent to the complex equality $s_2+\ii t_1=t_2+\ii s_1$, clearly implying $s_2=t_2$ and $s_1=t_1$.
	
	This shows that, for all $(s_1,s_2)\in S$ and $(t_1,t_2)\in\mathbb{N}^2$, excluding the trivial cases, the function~\eqref{eq:resonanceta} cannot be identically zero. Now take $0\leq\eta\leq\eta_0$ for some $\eta_0>0$ and consider the first $n$ analytic eigenvalue functions, $\{E_i^{\lambda,\delta}\}_{i=1}^n$. There is only a finite number of possible resonance functions $f^{\lambda,\delta}_{s_1,s_2,t_1,t_2}(\eta)$ among them and therefore only a finite number of values of $\eta$ where some of them, possibly more than one, vanish. Denote the complement of this set by $W^{\lambda,\delta}_n$, which is a dense open subset of $[0,\eta_0]$. The intersection $W^{\lambda,\delta}=\bigcap_{n\in\mathbb{N}}W^{\lambda,\delta}_n$ is dense because $[0,\eta_0]$ is a Baire space, and therefore $W^{\lambda,\delta}\cap[0,\eta_0] \neq \emptyset$, which proves that there exists $\eta\in[0,\eta_0]$ without resonances. Since $\eta_0$ was arbitrary, $\eta$ can be chosen as small as needed.
	
	It remains to show that $\eta$ can be chosen in such a way that $S$ is still a connectedness chain. Let $j,\ell\in\mathbb{N}^2$; since $S$ is a connectedness chain for the original control system, there is a finite sequence $(j,r_1),(r_1,r_2),\ldots,(r_{k-1},r_k),(r_k,\ell)\in S$ such that
	\begin{equation}
		\Braket{\varphi_j,V\varphi_{r_1}},\;\Braket{\varphi_{r_1},V\varphi_{r_2}},\;\ldots,\;\Braket{\varphi_{r_{k-1}},V\varphi_{r_k}},\;\Braket{\varphi_{r_{k}},V\varphi_{\ell}}\neq0.
	\end{equation}
	But then, again by an analyticity argument analogous to the one above,
	\begin{equation}
		\Braket{\varphi^{\lambda,\delta}_j(\eta),V\varphi^{\lambda,\delta}_{r_1}(\eta)}\cdots\Braket{\varphi^{\lambda,\delta}_{r_{k-1}}(\eta),V\varphi^{\lambda,\delta}_{r_k}(\eta)}\Braket{\varphi^{\lambda,\delta}_{r_{k}}(\eta),V\varphi^{\lambda,\delta}_{\ell}(\eta)}\neq0
	\end{equation}
	for all but countably many values of $\eta$. Indeed, the function above is analytic in $\eta$ \textit{and} non-zero at $\eta=0$. There are finitely many zeros for $0\leq\eta\leq\eta_0$, and therefore the function above is different from zero in an open neighborhood of $\eta=0$. Consequently, $S$ is a connectedness chain in that neighborhood, thus proving the statement.
\end{proof}

\begin{remark}\label{rem:lambda0}
	In the case $\lambda=0$, i.e.\ when the control system is $\Delta_{\rm Dir}+u\delta\,p,\,u\in(-r,r)$, the argument in Theorem~\ref{thm:auxiliary} does not apply directly since the derivatives of the perturbed eigenvalues at $\eta=0$ (see Eq.~\eqref{eq:secondder}) are all equal up to the second order. In fact, in this case it is fairly easy to prove a negative result: resonances \textit{cannot} be removed via the analyticity argument above. Indeed, the eigenvalue problem for the operator $\Delta_{\rm Dir}+\eta\delta\,p$ corresponds to the following equation:
	\begin{equation}
		\begin{cases}	
			-\Psi''(x)+\ii\,\eta\,\delta\Psi'(x)=E\Psi(x);\\
			\Psi\left(\pm\frac{1}{2}\right)=0,
		\end{cases}
	\end{equation}
	which can be easily solved explicitly, and admits solutions if and only if $\sin\left[\sqrt{2E+\frac{\delta^2\eta^2}{2}}\right]=0$, that is,
	\begin{equation}
		E_j(\eta)=j^2\pi^2-\frac{1}{4}\delta^2\eta^2,\qquad j\geq1,
	\end{equation}
	in agreement with Eq.~\eqref{eq:secondder}. Therefore, the perturbation only affects the spectrum by a uniform shift, and resonances cannot be removed in this way.
\end{remark}
Lemma~\ref{lemma:nonresonant} leads us to the proof of Theorem~\ref{thm:auxiliary}:

\begin{proof}[Proof of Theorem~\ref{thm:auxiliary}]
	
	We have to verify that $\Delta_{\rm Dir}+\eta V^{\lambda,\delta}$, $\eta\in(-r,r)$, satisfies the assumptions of~\cite[Theorem 2.6]{boscain2012weak}. As discussed, the free Hamiltonian $\Delta_{\rm Dir}$ has a purely discrete and non-degenerate spectrum, with a complete basis of eigenvectors. Moreover, in the case $\delta\neq0$ and $\lambda>0$, by Lemma~\ref{lemma:nonresonant} we can always find $\eta\in(-r,r)$ such that the perturbed control system has a non-resonant connectedness chain and thus, by~\cite[Theorem 2.6]{boscain2012weak}, is approximately controllable. Therefore, by Prop.~\ref{rem:controlperturb}, the original control system is approximately controllable as well.
	
	Now let $\delta=0$ and $\lambda>0$, which leads us to consider (up to an immaterial constant that can be reabsorbed by an immediate scaling) the control system $\Delta_{\rm Dir}+ux\circ p,\,u\in(-r,r)$. As already noticed,~\cite[Theorem 2.6]{boscain2012weak} does not apply in such a case since no connectedness chain can be found at all. However, it is easy to show that, by splitting the Hilbert space $L^2(\Omega)$ into its two parity-invariant sectors, then both operators $\Delta_{\rm Dir}$ and $x\circ p$ admit $L^2_\pm(\Omega)$ as reducing subspaces. This again implies that the two parity sectors $L^2_\pm(\Omega)$ of the control system $\Delta_{\rm Dir}+ux\circ p,\,u\in(-r,r)$ are separately approximately controllable via piecewise constant functions.
\end{proof}

\subsection{Step two: weak controllability by piecewise smooth functions}

We will now transfer the controllability result for the auxiliary system~\eqref{eq:auxiliary} via piecewise constant functions, given by Theorem~\ref{thm:auxiliary}, to the system~\reftildesigma{} in Eq.~\eqref{eq:schro2} with controls determined as in~\eqref{eq:transf2}. We will need to make use of the norms $\|\cdot\|_\pm$ as defined in Eq.~\eqref{eq:plusminus}, with $H_0$ being, in the case of interest here, the Dirichlet Laplacian $\Delta_{\rm Dir}$, and make use of the stability property given by Theorem~\ref{thm:abstract2}. 

\begin{theorem}\label{thm:piecewise}
	Let $\lambda>0$, $\delta\in\mathbb{R}$ and $r>0$, Then,
	\begin{itemize}
		\item[(i)] if $\delta\neq0$, for all $\Psi_0,\Psi_1\in L^2(\Omega)$ with $\|\Psi_0\|=\|\Psi_1\|$, and all $\epsilon>0$, there exist $T>0$ and a piecewise linear function $f:[0,T]\rightarrow\mathbb{R}$, with $|f(t)|<\ell_0/\lambda$, $|\dot f(t)|<r$ a.e., $f(0)=0$, and such that there exists an admissible solution $\tilde{U}^{\lambda,\delta}_f(t,s)$, $t,s\in [0,T]$, of system~\reftildesigma{} that satisfies
		\begin{equation}\label{eq:control}
			\left\|\Psi_1-\tilde{U}^{\lambda,\delta}_f(T,0)\Psi_0\right\|_-<\epsilon;
		\end{equation}
		\item[(ii)] if $\delta=0$, for all $\Psi_0,\Psi_1\in L^2_+(\Omega)$ (or equivalently $\Psi_0,\Psi_1\in L^2_-(\Omega)$) with $\|\Psi_0\|=\|\Psi_1\|$, and all $\epsilon>0$, there exist $T>0$ and a piecewise linear function $f:[0,T]\rightarrow\mathbb{R}$, with $|f(t)|<\ell_0/\lambda$, $|\dot f(t)|<r$ a.e., $f(0)=0$, and such that there exists an admissible solution $\tilde{U}^{\lambda,\delta}_f(t,s)$, $t,s\in [0,T]$, of system~\reftildesigma{} that satisfies Eq.~\eqref{eq:control}.
	\end{itemize}
\end{theorem}

\begin{proof}
	We will prove (i); (ii) follows identically.
	
	Let $\Psi_0,\Psi_1\in H^1_0(\Omega)=\hilb^+$ with $\|\Psi_0\|=\|\Psi_1\|$, and $\epsilon>0$. 
	By Theorem~\ref{thm:auxiliary}, there exist $T>0$ and a piecewise constant function $v:[0,T]\rightarrow\mathbb{R}$, with $|v(t)|<r$, such that
	\begin{equation}\label{eq:start}
		\|\Psi_1-\hat{U}_v^{\lambda,\delta}(T,0)\Psi_0\|<\frac{\epsilon}{2},
	\end{equation}
	where $\hat{U}_v^{\lambda,\delta}(t,s)$, $t,s \in [0,T]$ is by Proposition~\ref{prop:transform}, cf.\ Remark~\ref{rem:admissible}, an admissible solution of the Schrödinger equation associated with the Hamiltonian $\hat{H}^{\lambda,\delta}_v(t)$.
	
	Let $d>0$ be the number of constant pieces of $v(t)$. We will now construct a sequence of piecewise linear functions $\{f_n\}_{n\in\mathbb{N}}$ such that, for all $n\geq1$, $\dot{f}_n(t)=v(t)$ at every continuity point of $f_n(t)$. For each $n\geq1$, divide the interval $[0,T]$ into $n$ subintervals of equal length $T/n$, and let $\{I_{n,j}\}_{j}$ be the coarsest refinement of this partition such that $v(t)$ is constant in each interval $I_{n,j}$, with $I_{n,j}=[t_{n,j},t_{n+1,j}]$. This partition is made of $\tilde{n}\leq n+d$ subintervals (one possibly added interval for each singularity of $v(t)$). With this construction, we define
	\begin{equation}\label{eq:fn}
		f_n(t)=\sum_{j=1}^{\tilde{n}}\left(\int_{t_{n,j}}^t v(s)\,\mathrm{d}s\right)\chi_{I_{n,j}}(t),
	\end{equation}
	with $\chi_{I_{n,j}}(t)$ being the characteristic function of the interval $I_{n,j}$. This is a piecewise linear function whose derivative, for all values of $t\neq t_{n,j}$, $j=1,\dots,\tilde{n}$, equals $\dot{f}_n(t)=v(t)$. In each interval $I_{n,j}$, it starts from the zero value and grows linearly. In particular, for all $t\in [0,T]$, Eq.~\eqref{eq:fn} implies
	\begin{equation}\label{eq:bound0}
		\left|f_n(t)\right|<\frac{rT}{n},
	\end{equation}
	since $|v(t)|<r$ and the length of each interval $I_{n,j}$ is at most $T/n$. In particular, 
	fixing any $\mu>0$, for $n$ sufficiently large we have $|f_n(t)|\leq\mu$; we shall henceforth restrict to such values of $n$.
		
	The basic idea is the following: $f_n(t)$ becomes closer to $0$ as $n$ grows, while its derivative is still $\dot f_n(t)=v(t)$ whenever it exists, so that, for large $n$, the auxiliary Hamiltonian $\hat{H}^{\lambda,\delta}_v(t)$ and the Hamiltonian $\tilde{H}^{\lambda,\delta}_{f_n}(t)$, 
	respectively defined by Eq.~\eqref{eq:auxiliary} and Eq.~\eqref{eq:transf2} with $\dot{f}_n(t)=v(t)$,
	become closer, and so will their corresponding unitary propagators.
	Since $\lambda$, $\delta$, and $v$ are fixed, in the remainder of the proof we shall simplify our notation by simply setting
	\hbox{$	\hat{U}_v^{\lambda,\delta}(t,s)\equiv U_0(t,s)$,}
	\hbox{$U^{\lambda,\delta}_{f_n}(t,s)\equiv U_n(t,s)$},
	and correspondingly \hbox{$\hat{H}_v^{\lambda,\delta}(t)\equiv H_0(t)$}, \hbox{$\tilde{H}^{\lambda,\delta}_{f_n}(t)\equiv H_n(t)$}, and $V^{\lambda,\delta}\equiv V$.
	
	Fix $n\geq1$, $j=1,\dots,\tilde{n}$, and consider the \textit{pair} of Hamiltonians $H_0(t)$, $H_{n}(t)$, with $t$ ranging in the $j$th interval $I_{n,j}$. We aim at using the stability argument (Theorem~\ref{thm:abstract2}), for a \textit{fixed} value of $n>\lambda rT/\ell_0$, to the $n$-dependent pair of Hamiltonians $H_0(t)$, $H_{n}(t)$, in the $n$- and $j$-dependent interval $I_{n,j}$. Notice that we cannot invoke any stability argument for the full family of Hamiltonians $\{\tilde{H}_{n}(t)\}_{n}$, since, for sufficiently large $n$, any fixed time interval will contain singularities of $f_n(t)$.
	
	Since $V$ is a linear combination of $p$ and $x\circ p$, by Lemma~\ref{lemma:relbounded} it is infinitesimally form bounded with respect to $\Delta_{\rm Dir}$, so that Hypothesis~\ref{hyp} is satisfied. Since all functions are two times continuously differentiable and bounded, Theorem~\ref{thm:abstract2} applies:
	\begin{equation}\label{eq:boundlnj}
		\left\|U_{n}(t,s)-U_0(t,s)\right\|_{+,-}\leq L_{n,j}\left\|G_n\right\|_{L^1(I_{n,j})},
	\end{equation}
	where
	\begin{equation}\label{eq:gn}
		G_n(t):=\left(\frac{1}{\ell_0^2}-\frac{1}{\left[\ell_0+\lambda f_n(t)\right]^2}\right)+v(t)\left(\frac{1}{\ell_0}-\frac{1}{\ell_0+\lambda f_n(t)}\right),
	\end{equation}
	and with a constant $L_{n,j}$ being given by Eq.~\eqref{eq:l}, which is therefore proportional to the exponential of the largest of the $L^1(I_{n,j})$ norms of the derivatives of the coefficients of the two Hamiltonians in Eqs.~\eqref{eq:auxiliary} and~\eqref{eq:transf2}. The two only coefficients with non-zero derivatives satisfy, recalling again that $\dot f_n(t)=v(t)$ and $\dot{v}(t)=0$ in $I_{n,j}$,
	\begin{eqnarray}
		\left|\frac{\mathrm{d}}{\mathrm{d}t}\frac{1}{\left[\ell_0+\lambda f_n(t)\right]^2}\right|&=&\frac{2\lambda v(t)}{\left[\ell_0+\lambda f_n(t)\right]^3}\leq\frac{2\lambda r}{\left[\ell_0-\lambda\mu\right]^3}\\
		\left|\frac{\mathrm{d}}{\mathrm{d}t}\frac{v(t)}{\ell_0+\lambda f_n(t)}\right|&=&\frac{\lambda v(t)^2}{\left[\ell_0+\lambda f_n(t)\right]^2}\leq\frac{\lambda r^2}{\left[\ell_0-\lambda\mu\right]^2},
	\end{eqnarray}
	thus implying
	\begin{eqnarray}\label{eq:lnj}
		L_{n,j}&\leq&c^8\exp\left\{2Kc^2\max\left\{\frac{2T\lambda r}{n\left[\ell_0-\lambda\mu\right]^3},\frac{T\lambda r^2}{n\left[\ell_0-\lambda\mu\right]^2}\right\}\right\}\nonumber\\&\leq& c^8\exp\left\{2Kc^2\max\left\{\frac{2T\lambda r}{\left[\ell_0-\lambda\mu\right]^3},\frac{T\lambda r^2}{\left[\ell_0-\lambda\mu\right]^2}\right\}\right\}=:L,
	\end{eqnarray}
	with the constants $K>0$ and $c>1$ as given respectively in Lemma~\ref{lemma:kappa} and Theorem~\ref{thm:abstract2}(a). $L$ can be therefore been taken common to \textit{all} $n$ and $j$. This can be done since the equivalence of norms in the theorem does not require any regularity property on the coefficients, and can thus be applied to all Hamiltonians.
	
	This shows that Eq.~\eqref{eq:boundlnj} holds, in fact, with a constant $L$ independent of $n$ or $j$:
	\begin{equation}\label{eq:boundl}
		\left\|U_{n}(t,s)-U_0(t,s)\right\|_{+,-}\leq L\left\|G_n\right\|_{L^1(I_{n,j})}.
	\end{equation}
	Let us estimate $\|G_n\|_{L^1(I_{n,j})}$. By Eq.~\eqref{eq:gn}, and recalling Eq.~\eqref{eq:bound0},
	\begin{eqnarray}
		G_n(t)&=&\frac{2\ell_0\lambda f_n(t)+\lambda^2f_n(t)^2}{\ell_0^2\left[\ell_0+\lambda f_n(t)\right]^2}+v(t)\frac{\lambda f_n(t)}{\ell_0\left[\ell_0+\lambda f_n(t)\right]}\nonumber\\
		&\leq&\frac{2\lambda rT}{n\ell_0\left[\ell_0-\lambda\mu\right]^2}+\frac{\lambda^2r^2T^2}{n^2\ell_0^2\left[\ell_0-\lambda\mu\right]}+\frac{\lambda r^2T}{n\ell_0\left[\ell_0-\lambda\mu\right]},
	\end{eqnarray}
	and therefore
	\begin{equation}\label{eq:injbound}
		\|G_n\|_{L^1(I_{n,j})}\leq
		\frac{2\lambda rT^2}{n^2\ell_0\left[\ell_0-\lambda\mu\right]^2}
		+\frac{\lambda^2r^2T^3}{n^3\ell_0^2\left[\ell_0-\lambda\mu\right]^2}
		+\frac{\lambda r^2T^2}{n^2\ell_0\left[\ell_0-\lambda\mu\right]}
		=\mathcal{O}(n^{-2})	.
	\end{equation}
	The equation above implies that, separately in each interval $I_{n,j}$, the norm of the difference between the two propagators is $\mathcal{O}(n^{-2})$ and independent of $j$. Therefore, we have
	\begin{eqnarray}\label{eq:claim}
		\left\|\Psi_1-U_{n}(T,0)\Psi_0\right\|_-&\leq&\|\Psi_1-U_0(T,0)\Psi_0\|_-+\|U_{n}(T,0)\Psi_0-U_0(T,0)\Psi_0\|_-\nonumber\\
		&<&\|\Psi_1-U_0(T,0)\Psi_0\|+\|U_{n}(T,0)\Psi_0-U_0(T,0)\Psi_0\|_-\nonumber\\
		&<&\frac{\epsilon}{2}+\|U_{n}(T,0)\Psi_0-U_0(T,0)\Psi_0\|_-,
	\end{eqnarray}
	hence it is enough to show that, for every $t\in[0,T]$,
	\begin{equation}\label{eq:will}
		\lim_{n\to\infty}\|U_{n}(t,0)\Psi_0-U_0(t,0)\Psi_0\|_-=0
	\end{equation}
	uniformly in $t$. Let us set hereafter
	\begin{equation}
		\Psi_{0,n}(t)=U_{n}(t,0)\Psi_0,\qquad \Psi_0(t)=U_0(t,0)\Psi_0
	\end{equation}
	for all $t$. The family of norms associated with the sequence of time-dependent Hamiltonians $\{H_n(t)\}_{n\in\mathbb{N}}$ will be labeled by $\|\cdot\|_{\pm,n,t}$. Recall that by Theorem~\ref{thm:abstract2}(a)
	\begin{equation}\label{eq:equiv}
		\frac{1}{c}\|\cdot\|_{\pm}\leq\|\cdot\|_{\pm,n,t}\leq c\|\cdot\|_{\pm}.
	\end{equation}
	Suppose $t\in I_{n,j}$ for some integer $j$. Then
	\begin{eqnarray}
		\left\|\Psi_{0,n}(t)-\Psi_0(t)\right\|_{-,0,t}&=&\left\|U_{n}(t,t_{n,j})\Psi_{0,n}(t_{n,j})-U_0(t,t_{n,j})\Psi_{0}(t_{n,j})\right\|_{-,0,t}\nonumber\\
		&\leq&\left\|\left(U_n(t,t_{n,j})-U_0(t,t_{n,j})\right)\Psi_{0,n}(t_{n,j})\right\|_{-,0,t}\nonumber\\
		&&+\left\|U_0(t,t_{n,j})\left(\Psi_{0,n}(t_{n,j})-\Psi_0(t_{n,j})\right)\right\|_{-,0,t}.
	\end{eqnarray}
	The first term can be bounded by using Eq.~\eqref{eq:boundl} and the equivalence of norms~\eqref{eq:equiv}:
	\begin{eqnarray}\label{eq:firstterm}
		\left\|\left(U_n(t,t_{n,j})-U_0(t,t_{n,j})\right)\Psi_{0,n}(t_{n,j})\right\|_{-,0,t}&\leq&c\left\|\left(U_n(t,t_{n,j})-U_0(t,t_{n,j})\right)\Psi_{0,n}(t_{n,j})\right\|_{-}\nonumber\\
		&\leq&c\left\|U_n(t,t_{n,j})-U_0(t,t_{n,j})\right\|_{+,-}\|\Psi_{0,n}(t_{n,j})\|_{+}\nonumber\\
		&\leq&cL\|G_n\|_{L^1(t_{n,j},t)}
		\|\Psi_{0,n}(t_{n,j})\|_+;
	\end{eqnarray}
	on the other hand, following the ideas in~\cite[Appendix II.7]{Simon1971}, it can be shown that
	\begin{equation}\label{eq:firstterm2}
		\|\Psi_{0,n}(t_{n,j})\|_{+,n,t}\leq 
		\e^{\frac{3}{2}c^2M}\|\Psi_0\|_{+,n,0},
	\end{equation}
	where the constant $M\geq0$ only depends on the derivatives of the coefficients of the form linear Hamiltonian and vanishes if they vanish. Therefore, by using Eqs.~\eqref{eq:firstterm}--\eqref{eq:firstterm2} and the equivalence of norms twice,
	\begin{equation}
		\left\|\left(U_n(t,t_{n,j})-U_0(t,t_{n,j})\right)\Psi_{0,n}(t_{n,j})\right\|_{-,0,t}\leq c^3L\e^{\frac{3}{2}c^2M}\|G_n\|_{L^1(t_{n,j},t)} \|\Psi_0\|_+.
	\end{equation}
	Besides, since $\dot v(t)=0$ for $t\in I_{n,j}$,
	\begin{equation}\label{eq:secondterm}
		\left\|U_0(t,t_{n,j})\left(\Psi_{0,n}(t_{n,j})-\Psi_0(t_{n,j})\right)\right\|_{-,0,t}\leq	\left\|\Psi_{0,n}(t_{n,j})-\Psi_0(t_{n,j})\right\|_{-,0,t_{n,j}}.
	\end{equation}
	Combining the inequalities~\eqref{eq:firstterm} and~\eqref{eq:secondterm} we get
	\begin{equation}
		\left\|\Psi_{0,n}(t)-\Psi_0(t)\right\|_{-,0,t}\leq c^3L\e^{\frac{3}{2}c^2M}\|G_n\|_{L^1(t_{n,j},t)} \|\Psi_0\|_++\left\|\Psi_{0,n}(t_{n,j})-\Psi_0(t_{n,j})\right\|_{-,0,t_{n,j}}.
	\end{equation}
	Therefore, we have bounded the difference between the evolutions at time $t$ with respect to the difference of the evolutions at time $t_{n,j}$. Iterating these bounds and applying Eq.~\eqref{eq:injbound} to each of the $\tilde{n}\leq n+d$ subintervals $I_{n,j}$, one obtains
	\begin{eqnarray}
		\left\|\Psi_{0,n}(t)-\Psi_0(t)\right\|_{-,0,t}&\leq& c^3L\e^{\frac{3}{2}c^2M} \|\Psi_0\|_+\Bigl(
		\|G_n\|_{L^1(t_{n,j},t)}+\|G_n\|_{L^1(t_{n,j-1},t_{n,j})}+\nonumber\\
		&&...+\|G_n\|_{L^1(0,t_{n,1})}\Bigr)\nonumber\\
		&\leq&c^3L\e^{\frac{3}{2}c^2M} \|\Psi_0\|_+
		(n+d)
		\Bigl(
		\frac{2\lambda rT^2}{n^2\ell_0^3}+\frac{\lambda^2r^2T^3}{n^3\ell_0^4}+\frac{\lambda r^2T^2}{n^2\ell_0^2}
		\Bigr),
	\end{eqnarray}\normalsize
	and the latter quantity goes to zero as $n\to\infty$, which, having into account the uniform equivalence of the norms, shows that Eq.~\eqref{eq:control} does hold for all $\Psi_0,\Psi_1\in\mathrm{H}^1_0(\Omega)=\hilb^+$.
	
	To complete the proof, we must show that Eq.~\eqref{eq:control} holds for arbitrary $\Psi_0,\Psi_1\in L^2(\Omega)$. To this purpose, we can simply make use of an $\epsilon/3$-argument. Namely, given $\Psi_0,\Psi_1\in L^2(\Omega)$ with $\|\Psi_0\|=\|\Psi_1\|$ and $\epsilon>0$, there exist $\Psi_0',\Psi_1'\in\hilb^+$ with $\|\Psi_0\|=\|\Psi_1\|$ such that
	\begin{equation}
		\|\Psi_s-\Psi_s'\|<\frac{\epsilon}{3},\qquad s=0,1,
	\end{equation}
	and, as proven above, there exist $T>0$ and a piecewise linear function $f:[0,T]\rightarrow\mathbb{R}$, such that
	\begin{equation}\label{eq:controlbis}
		\left\|\Psi_1'-\tilde{U}^{\lambda,\delta}_f(T,0)\Psi_0'\right\|_-<\frac{\epsilon}{3},
	\end{equation}
	whence, by the properties of the norms $\|\cdot\|$ and $\|\cdot\|_\pm$ (cf.\ Remark~\ref{rem:scale}),
	\begin{eqnarray}
		\left\|\Psi_1-\tilde{U}^{\lambda,\delta}_f(T,0)\Psi_0\right\|_-&<&\|\Psi_1-\Psi_1'\|_-+\left\|\Psi_1'-\tilde{U}^{\lambda,\delta}_f(T,0)\Psi_0'\right\|_-+\|\tilde{U}^{\lambda,\delta}_f(T,0)(\Psi_0-\Psi_0')\|_-\nonumber\\
		&\leq&\|\Psi_1-\Psi_1'\|+\left\|\Psi_1'-\tilde{U}^{\lambda,\delta}_f(T,0)\Psi_0'\right\|_-+\|\tilde{U}^{\lambda,\delta}_f(T,0)(\Psi_0-\Psi_0')\|\nonumber\\
		&<&\frac{\epsilon}{3}+\frac{\epsilon}{3}+\frac{\epsilon}{3}=\epsilon,
	\end{eqnarray}
	where we also used the unitarity of $\tilde{U}^{\lambda,\delta}_f(T,0)$.
\end{proof}

Theorem~\ref{thm:piecewise} is a first, important step towards our desired result. Some considerations are in order. First, notice that the control function $f:[0,T]\rightarrow\mathbb{R}$ is a piecewise linear function, given by Eq.~\eqref{eq:fn}, with a number of discontinuities $n$ which is larger for smaller $\epsilon>0$: in words, if we require approximate controllability with a higher degree of precision, the price to pay is a higher number of such discontinuities. 

Also notice that, by construction, we have $f(0)=0$: the initial value of the control function is known beforehand. As discussed before, and for reasons that will become clear when we prove the final result, we will need the \textit{final} value $f(T)$ to be known a priori as well. This can be achieved without effort: basically, we need to make the system evolve with the control function given by Theorem~\ref{thm:piecewise}, and then making it evolve for a sufficiently small time with a constant-valued control equal to the desired final value. This is the content of the next result.

\begin{corollary}\label{prop:piecewise2}
	Let $\lambda>0$, $\delta\in\mathbb{R}$, $r>0$, and $a>-\ell_0/\lambda$. Then,		
	\begin{itemize}
		\item[(i)] if $\delta\neq0$, for all $\Psi_0,\Psi_1\in L^2(\Omega)$ with $\|\Psi_0\|=\|\Psi_1\|$, and all $\epsilon>0$, there exist $T>0$ and a piecewise linear function $f:[0,T]\rightarrow\mathbb{R}$, with $|f(t)|<\ell_0/\lambda$, $|\dot f(t)|<r$ a.e., $f(0)=0$, $f(T)=a$, and such that there exists an admissible solution $\tilde{U}^{\lambda,\delta}_f(t,s)$, $t,s\in [0,T]$, of system~\reftildesigma{} that satisfies
		\begin{equation}\label{eq:control2}
			\left\|\Psi_1-\tilde{U}^{\lambda,\delta}_f(T,0)\Psi_0\right\|_-<\epsilon;
		\end{equation}
		\item[(ii)] if $\delta=0$, for all $\Psi_0,\Psi_1\in L^2_\pm(\Omega)$ with $\|\Psi_0\|=\|\Psi_1\|$, and all $\epsilon>0$, there exist $T>0$ and a piecewise linear function $f:[0,T]\rightarrow\mathbb{R}$, with $|f(t)|<\ell_0/\lambda$, $|\dot f(t)|<r$ a.e., $f(0)=0$, $f(T)=a$, and such that there exists an admissible solution $\tilde{U}^{\lambda,\delta}_f(t,s)$, $t,s\in [0,T]$, of system~\reftildesigma{} that satisfies Eq.~\eqref{eq:control2}.			
	\end{itemize}
\end{corollary}

\begin{proof}
	Again we prove (i), with (ii) following identically. Let $\Psi_0,\Psi_1\in L^2(\Omega)$ with $\|\Psi_0\|=\|\Psi_1\|$, and $\epsilon>0$. By Theorem~\ref{thm:piecewise}, there exist $T'>0$ and a piecewise linear function $g:[0,T']\rightarrow\mathbb{R}$ with $f(0)=0$ such that
	\begin{equation}\label{eq:control3}
		\left\|\Psi_1-\tilde{U}^{\lambda,\delta}_g(T',0)\Psi_0\right\|_-<\frac{\epsilon}{2}.
	\end{equation}
	Consider the time-independent Hamiltonian $\tilde{H}^{\lambda,\delta}_a$, cf.~Eq.~\eqref{eq:transf2}, with $a(t)\equiv a$ the constant function, that is,
	\begin{equation}
		\tilde{H}^{\lambda,\delta}_a=\frac{1}{\left[\ell_0+a\lambda\right]^2}\Delta_{\rm Dir}.
	\end{equation}
	The latter is a constant Hamiltonian, whence, by Stone's theorem, the solution of the Schrödinger equation generated by it is a strongly continuous one-parameter group. An $\epsilon/2$-argument proves that there exists $\tau>0$ such that
	\begin{eqnarray}
		\left\|\Psi_1-\e^{-\ii\tau \tilde{H}^{\lambda,\delta}_a}\tilde{U}_g^{\lambda,\delta}(T',0)\Psi_0\right\|_-<\epsilon,
	\end{eqnarray}
	which completes the proof.
\end{proof}

\subsection{Step three: from the auxiliary system to the moving walls setup}

\begin{proof}[Proof of Theorem~\ref{thm:main}]
	Let $\ell_0,\ell_1>0$, $d_0,d_1\in\mathbb{R}$.
	We will show that approximate controllability can be achieved with two functions $t\mapsto d(t),\ell(t)$ as given by Eqs.~\eqref{eq:dft}--\eqref{eq:lft} for some suitable choice of the parameters $\lambda,\delta$, and the control function $t\mapsto f(t)$. The latter must satisfy the following constraints at the final time $T>0$: 
	\begin{equation}
		\begin{cases}
			\lambda\,f(T)&=\Delta\ell;\\
			\delta\,f(T)&=\Delta d,
		\end{cases}
	\end{equation}
	with $\Delta\ell=\ell_1-\ell_0$ and $\Delta d=d_1-d_0$. 
	
	\paragraph{Case (i): $\delta \neq0$.} Without loss of generality we can assume $\lambda=1$. We will consider two measurable functions $t\mapsto \ell(t),d(t)$ in the form
	\begin{equation}\label{eq:delta}
			\ell(t)=\ell_0+f(t),\qquad d(t)=d_0+ \delta f(t),
    \end{equation}
	with $f(t)$ to be determined. For $\Delta\ell\cdot\Delta d\neq 0$  take $\delta=\Delta d/\Delta \ell$, for $\Delta\ell=\Delta d = 0$ take $\delta=1$.
	Define $\Psi_0=W_{\ell_0,d_0}\Phi_0$ and $\Psi_1=W_{\ell_1,d_1}\Phi_1$, with $W_{\ell,d}$ as defined in Eq.~\eqref{eq:w}. Let \hbox{$\Phi_1\in\mathrm{H}^1_0\left(\Omega_{\ell_1,d_1}\right)$}; since the transformation $W_{\ell,d}$ preserves the Sobolev spaces, see Eq.~\eqref{eq:regularity},
	\begin{equation}
		\Phi_1\in\mathrm{H}^1_0\left(\Omega_{\ell_1,d_1}\right)\implies \Psi_1=W_{\ell_1,d_1}\Phi_1\in\mathrm{H}^1_0(\Omega)=\hilb^+.
	\end{equation}
	By Corollary~\ref{prop:piecewise2}(i), there exist $T>0$ and a piecewise linear function $f:[0,T]\rightarrow\mathbb{R}$ satisfying
	\begin{equation}\label{eq:final}
			f(0)=0,\qquad f(T)=\Delta\ell,\qquad |f(t)|<\ell_0,\qquad |\dot f(t)|<r\;\;  \text{a.e.}
		\end{equation}
	and such that there exists an admissible solution $\tilde{U}^{1,\delta}_f(t,s)$, $t,s\in[0,T]$ of system~\reftildesigma{} that satisfies
	\begin{equation}
		\left\|\Psi_1-\tilde{U}^{1,\delta}_f(T,0)\Psi_0\right\|_-<\frac{\epsilon^2}{2\|\Psi_1\|_+}.
	\end{equation}
	Consequently,
	{\small
			\begin{eqnarray}
				\left\|\Psi_1-\tilde{U}^{1,\delta}_f(T,0)\Psi_0\right\|^2&=&\Braket{\Psi_1,\Psi_1-\tilde{U}^{1,\delta}_f(T,0)\Psi_0}+\Braket{-\tilde{U}^{1,\delta}_f(T,0)\Psi_0,\Psi_1-\tilde{U}^{1,\delta}_f(T,0)\Psi_0}\nonumber\\
				&=&\Braket{\Psi_1,\Psi_1-\tilde{U}^{1,\delta}_f(T,0)\Psi_0}+\|\Psi_0\|^2+\Braket{\Psi_1-\tilde{U}^{1,\delta}_f(T,0)\Psi_0,\Psi_1}-\|\Psi_1\|^2\nonumber\\
				&=&2\Re\,\Braket{\Psi_1,\Psi_1-\tilde{U}^{1,\delta}_f(T,0)\Psi_0}\nonumber\\
				&\leq&2\left|\Braket{\Psi_1,\Psi_1-\tilde{U}^{1,\delta}_f(T,0)\Psi_0}\right|\nonumber\\
				&\leq&2\|\Psi_1\|_+\left\|\Psi_1-\tilde{U}^{1,\delta}_f(T,0)\Psi_0\right\|_-<\epsilon^2,
			\end{eqnarray}
	}
	where we have used the properties of the norms $\|\cdot\|_{\pm}$, cf\ Remark~\ref{rem:scale}.
	
	On the other hand, with this choice of control function, the functions $t\mapsto\ell(t),d(t)$ defined via Eqs.~\eqref{eq:dft}--\eqref{eq:lft} satisfy, because of Eqs.~\eqref{eq:delta} and~\eqref{eq:final},
	\begin{equation}
		\ell(T)=\ell_1,\qquad d(T)=d_1;
	\end{equation}
	Moreover, we can use Corollary~\ref{coroll:transform} to obtain an admissible solution $U^{1,\delta}_f(t,s)$ of system~\refsigma{}. We have
	\begin{eqnarray}
			\tilde{U}^{1,\delta}_f(T,0)&=&W_{\ell_1,d_1}U^{1,\delta}_f(T,0)W^\dag_{\ell_0,d_0}.
		\end{eqnarray}
	By the unitarity of $W_{\ell,d}$,
	\begin{equation}
		\bigl\|\Phi_1-U^{1,\delta}_f(T,0)\Phi_0\bigr\|=\bigl\|W_{\ell_1,d_1}(\Phi_1-U^{1,\delta}_f(T,0)\Phi_0)\bigr\|
		=\bigl\|\Psi_1-\tilde{U}^{1,\delta}_f(T,0)\Psi_0\bigr\|
		<\epsilon.
	\end{equation}
	This inequality has been proven for $\Phi_1\in\mathrm{H}^1_0(\Omega_{\ell_1,d_1})$; since the latter is dense in $L^2(\Omega_{\ell_1,d_1})$, the inequality readily extends to an arbitrary $\Phi_1\in L^2(\Omega_{\ell_1,d_1})$ via an $\epsilon/2$-argument, thus proving the claim.
		
	\paragraph{Case (ii): $\delta=0$.} In this case we will consider two functions $t\mapsto \ell(t),d(t)$ in the form
	\begin{equation}
		\ell(t)=\ell_0+f(t),\qquad d(t)=d_0,
	\end{equation}
	that is, a pure dilation. This entails taking into account a wall motion as in Eqs.~\eqref{eq:dft}--\eqref{eq:lft} with $\lambda=1$ and $\delta=0$. The rest of the proof is identical to the case (i) but using Corollary~\ref{prop:piecewise2}(ii) instead.
\end{proof}

\section{Conclusions}\label{sec:conclusions}

In this work we have addressed the feasibility of quantum control for a non-relativistic particle confined in a moving box. Our main findings can be summarized in the following way. After preparing the system in any state confined in a box with given length and given position of the center, there always exists a particular motion of the box walls such that the initial state is driven arbitrarily close to any target state of a prescribed final position of the walls of the box, provided that the position of the center during the evolution does not remain constant. If, instead, the center of the box remains fixed during the evolution, the same result is achieved at the additional price of introducing a selection rule: the initial and target state must have the same, definite, parity. A natural continuation of this work is to find explicit solutions fo the control problem, in particular, addressing the optimal control problem. Numerical schemes like those presented in~\cite{ibort2013numerical, lopez2017finite} are well-suited for this purpose.

As stated in the introduction, a transformation for the problem of rigid translations of the box introduced in~\cite{rouchon2003control,beauchard2006controllability}, different than the one considered here, was used to prove local exact controllability in neighborhoods of the eigenstates. This transformation converts the problem of the Laplace operator in the rigid moving box into a problem of a fixed operator under the action of a time-dependent and uniform electric field. The results in~\cite{chambrion2009controllability} show that this system, which is a bilinear control problem, is approximately controllable with piecewise constant controls. As a straightforward consequence of the stability result, Theorem~\ref{thm:abstract2}, it can be shown that the controls can be chosen to be continuously differentiable and therefore the corresponding solution of the control problem for the rigid moving box is indeed a strongly differentiable solution of the Schrödinger equation. It remains an open problem to see if this regularity of the solutions can also be achieved for the general movement of the box considered here. We will address this problem in the future.

In the recent research~\cite{duca2021} it is shown that the Laplace operator on a simply connected domain in $\mathbb{R}^n$ which is varying can be mapped to a situation in which one has a fixed domain and an operator which is time-dependent and of the type of a magnetic Laplacian. This has been also studied with less generality in~\cite{anza2015}. We believe that the results presented in this article and in~\cite{balmaseda2021controllability, ibort2015self} could be used together with these recent descriptions to extend the present results about global approximate controllability of the Schrödinger equation in moving domains to base manifolds of arbitrary dimension, that is, approximate controllability with admissible solutions of the Schrödinger equation, cf. Definition~\ref{def:admissible}. These would constitute a great improvement with respect to the current controllability results that were discussed in the introduction. 

\section*{Acknowledgments}
{\small
	A.B. and J.M.P.P. acknowledge support provided by the ``Agencia Estatal de Investigación (AEI)'' Research Project PID2020-117477GB-I00, by the QUITEMAD Project P2018/TCS-4342 funded by the Madrid Government (Comunidad de Madrid-Spain) and by the Madrid Government (Comunidad de Madrid-Spain) under the Multiannual Agreement with UC3M in the line of ``Research Funds for Beatriz Galindo Fellowships'' (C\&QIG-BG-CM-UC3M), and in the context of the V PRICIT (Regional Programme of Research and Technological Innovation).
	J.M.P.P acknowledges financial support from the Spanish Ministry of Science and Innovation, through the ``Severo Ochoa Programme for Centers of Excellence in R\&D'' (CEX2019-000904-S).
	A.B. acknowledges financial support by ``Universidad Carlos III de Madrid'' through Ph.D. program grant PIPF UC3M 01-1819, UC3M mobility grant in 2020 and the EXPRO grant No. 20-17749X of the Czech Science Foundation.
	D.L. was partially supported by ``Istituto Nazionale di Fisica Nucleare'' (INFN) through the project ``QUANTUM'' and the Italian National Group of Mathematical Physics (GNFM-INdAM), and acknowledges support by MIUR via PRIN 2017 (Progetto di Ricerca di Interesse Nazionale), project QUSHIP (2017SRNBRK).
	He also thanks the Department of Mathematics at ``Universidad Carlos III de Madrid'' for its hospitality.	
	\printbibliography
}

\end{document}